\documentclass[11pt,letterpaper]{article}

\usepackage{amsmath,amssymb,amsfonts,amsthm,mathrsfs,cases}
\usepackage{graphicx}
\usepackage[usenames,dvipsnames]{color}

\usepackage[margin=1in]{geometry}

\theoremstyle{plain}
\newtheorem{thm}{Theorem}%[section]
\newtheorem{lem}[thm]{Lemma}
\newtheorem{cor}[thm]{Corollary}

\theoremstyle{definition}

\theoremstyle{remark}
\newtheorem{claim}{Claim}

\newcommand{\real}{\ensuremath {\mathbb R} }
\newcommand{\ent}{\ensuremath {\mathbb Z} }
\newcommand{\nat}{\ensuremath {\mathbb N} }

\newcommand{\remove}[1] {}

\newcommand{\ex} {{\bf E}}
\newcommand{\pr} {{\bf Pr}}
\newcommand{\var} {{\bf Var}}

\newcommand{\eps}{\varepsilon}
\renewcommand{\phi}{\varphi}
\newcommand{\bfrac}[2]{\of{\frac{#1}{#2}}}
\newcommand{\of}[1]{\left( #1 \right)}

\DeclareMathOperator{\Bin}{Bin}

\DeclareMathOperator{\END}{END}
\DeclareMathOperator{\Bernoulli}{Bernoulli}
\DeclareMathOperator{\out}{out}
\DeclareMathOperator{\old}{old}

\title{Perfect matchings and Hamiltonian cycles\\ in the preferential attachment model}
\author{Alan Frieze\thanks{Department of Mathematical Sciences, Carnegie Mellon University, Pittsburgh, PA, USA, e-mail: \texttt{alan@random.math.cmu.edu}; the author is supported in part by NSF grant DMS1362785.}
\and Xavier P\'erez-Gim\'enez\thanks{Department of Mathematics, University of Nebraska, Lincoln, NE, USA, e-mail: \texttt{xperez@unl.ca}}
\and Pawe{\l} Pra{\l}at\thanks{Department of Mathematics, Ryerson University, Toronto, ON, Canada and The Fields Institute for Research in Mathematical Sciences, Toronto, ON, Canada, e-mail: \texttt{pralat@ryerson.ca}; the author is supported in part by NSERC.}
\and Benjamin Reiniger\thanks{Department of Applied Mathematics, Illinois Institute of Technology, Chicago, IL, USA, e-mail: \texttt{breiniger@iit.edu}}
}
\date{}

\begin{document}
\maketitle

\begin{abstract}
In this paper, we study the existence of perfect matchings and Hamiltonian cycles in the preferential attachment model. In this model, vertices are added to the graph one by one, and each time a new vertex is created it establishes a connection with $m$ random vertices selected with probabilities proportional to their current degrees. (Constant $m$ is the only parameter of the model.) We prove that if $m \ge 1{,}260$, then asymptotically almost surely there exists a perfect matching.  Moreover, we show that there exists a Hamiltonian cycle asymptotically almost surely, provided that $m \ge 29{,}500$. One difficulty in the analysis comes from the fact that vertices establish connections only with vertices that are ``older'' (i.e.~are created earlier in the process).
However, the main obstacle arises from the fact that edges in the preferential attachment model are not generated independently. 
In view of that, we also consider a simpler setting---sometimes called uniform attachment---in which vertices are added one by one and each vertex connects to $m$ older vertices selected uniformly at random and independently of all other choices. We first investigate the existence of perfect matchings and Hamiltonian cycles in the uniform attachment model, and then extend the argument to the preferential attachment version. 
\end{abstract}

\section{Introduction}\label{sec:intro}

Two of the most natural questions concerning models of sparse random graph are whether or not a random instance is Hamiltonian or has a perfect matching. The hamiltonicity threshold for the basic models of {\em random graphs} $G(n,m)$ and $G(n,p)$ was established quite precisely by Koml\'os and Szemer\'edi~\cite{KS15}. (See Bollob\'as~\cite{B3}, Ajtai, Koml\'os and Szemer\'edi~\cite{AKS1} and Bollob\'as and Frieze~\cite{BF4} for refinements.) Perfect matchings were first investigated (using Tutte's theorem) by Erd\H{o}s and R\'enyi~\cite{Book279}. (A simpler argument using expansion properties was provided by Bollob\'as and Frieze~\cite{BF4}.) 

\smallskip

The existence of perfect matchings and Hamiltonian cycles has also been investigated for other relevant models of random graphs. For instance, D\'\i az, Mitsche and P\'erez-Gim\'enez~\cite{DMP07} obtained the hamiltonicity threshold for the {\em random geometric graph} $G_{n,r}$, which is defined as follows: the vertices of $G_{n,r}$ are $n$ points chosen independently and uniformly at random from the unit square, and any two vertices are joined by an edge if and only if they are within distance $r$ from each other in the square. The result in~\cite{DMP07} was further strengthened by Balogh, Bollob\'as, Krivelevich, M\"uller and Walters~\cite{BBKMW11} and M\"uller, P\'erez-Gim\'enez and Wormald~\cite{MPW11}, who gave a more precise characterization of the emergence of the first Hamiltonian cycle in the random geometric graph. (The analogous result for perfect matchings follows from a more general statement in~\cite{MPW11}.)

\smallskip

Another well-studied sparse random graph is the \emph{random regular graph}. Let $G_{n,d}$ denote a graph chosen uniformly at random from the set of $d$-regular graphs with vertex set $[n]$. Robinson and Wormald~\cite{RW18} showed that \emph{asymptotically almost surely} (a.a.s.) $G_{n,d}$ is Hamiltonian for constant $d \ge 3$. (Here and in similar statements, an event occurs a.a.s.\ if it occurs with probability tending to $1$ as $n$ tends to infinity.) Allowing $d$ to grow with $n$ presented some challenges, but they have now been resolved (see Cooper, Frieze and Reed~\cite{CFR9} and Krivelevich, Sudakov, Vu and Wormald~\cite{KSVW16}).

\smallskip

Yet another random graph model extensively studied from that perspective is the so-called \emph{$m$-out graph}. This time, each vertex $v \in [n]$ independently chooses $m$ random out-neighbours to create the random digraph $D^n_{m-\out}$. We then obtain $G^n_{m-\out}$ by ignoring orientations. The hamiltonicity of $G^n_{m-\out}$ was first discussed by Fenner and Frieze~\cite{FF11}. They showed that $G^n_{23-\out}$ is a.a.s.\ Hamiltonian. This was improved to $G^n_{10-\out}$ by Frieze~\cite{F12} and to $G^n_{5-\out}$ by Frieze and \L{}uczak~\cite{FL13}. Cooper and Frieze~\cite{CF8} showed that $G^n_{4-\out}$ is a.a.s.\ Hamiltonian, and the last drop was squeezed by Bohman and Frieze who established that $G^n_{3-\out}$ is a.a.s.\ Hamiltonian~\cite{BF}.

\smallskip

In this paper, we study the \emph{preferential attachment model}, which, arguably, is the best-known model for complex networks. The first consideration of this model goes back to 1925 when Yule used it to explain the power-law distribution of the number of species per genus of flowering plants~\cite{Yule}. The application of the model to describe the growth of the World Wide Web was proposed by Barab\'asi and Albert in 1999~\cite{BA}. We will use the following precise definition of the model, as considered by Bollob\'as and Riordan in~\cite{BR} as well as Bollob\'as, Riordan, Spencer, and Tusn\'ady~\cite{BRST}.

Let $G_1^1$ be the graph with one vertex, $1$, and one loop. The random graph process $(G_1^t)_{t \ge 1}$ is defined inductively as follows. Given $G_1^{t-1}$, we form $G_1^t$ by adding vertex $t$ together with a single edge between $t$ and $i$, where $i$ is selected randomly with the following probability distribution:
$$
\pr (i = s) =
\begin{cases}
\deg(s,t-1) / (2t-1) & 1 \le s \le t-1, \\
1/(2t-1) & s=t,
\end{cases}
$$
where  $\deg(s,t-1)$ denotes the degree of vertex $s$ in $G_1^{t-1}$ (i.e.~the degree of $s$ ``at time $t-1$'', after vertex $t-1$ was added). In other words, we send an edge $e$ from vertex $t$ to a random vertex $i$, where the probability that a vertex is chosen as $i$ is proportional to its degree at the time, counting $e$ as already contributing one to the degree of $t$. In particular, each vertex may only be attached to itself or to an ``older'' vertex (i.e.~a vertex created earlier in the process).

For $m \in \nat \setminus \{1\}$, the process $(G_m^t)_{t \ge 1}$ is defined similarly with the only difference that $m$ edges are added to $G_m^{t-1}$ to form $G_m^t$ (one at a time), counting previous edges as already contributing to the degree distribution. Equivalently, one can define the process $(G_m^t)_{t \ge 1}$ by considering the process $(G_1^t)_{t \ge 1}$ on a sequence $1', 2', \ldots$ of vertices; the graph $G_m^t$ is formed from $G_1^{tm}$ by identifying vertices $1', 2', \ldots, m'$ to form $1$, identifying vertices $(m+1)', (m+2)', \ldots, (2m)'$ to form $2$, and so on. Although $(G_m^t)_{t\ge1}$ is an infinite-time random process, we will restrict our attention to a finite number of time steps and consider $(G_m^t)_{1\le t\le n}$ or $G_m^n$ for large $n$. Note that in this model $G_m^n$ is in general a multigraph, possibly with multiple edges between two vertices (if $m\ge2$) and self-loops. For our purpose, as we are interested in Hamiltonian cycles and perfect matchings, loops can be ignored and multiple edges between two nodes can be treated as a single edge.

Perhaps the most remarkable feature of the preferential attachment mechanism is that it provides a simple illustration of the {\em rich-get-richer} principle, by which vertices that acquire higher degree early in the process use this accumulated advantage to attract more edges during the process. This results in a heavy-tailed degree sequence with some vertices of very high degree, which do not typically occur in the previous models of random graphs described above. Indeed, it was shown in~\cite{BRST} that, for any constant $m \in \nat$, a.a.s.\ the degree distribution of $G_m^n$ follows a power law: the number of vertices with degree $k$ falls off as $(1+o(1)) ck^{-3} n$  for some explicit constant $c=c(m)$ and large $k \le n^{1/15}$.

For the purpose of this paper, the case $m=1$ is easy to analyze, since $G_1^n$ is a forest. Each node sends an edge either to itself or to an earlier node, so the graph consists of components which are trees, each with a loop attached. The expected number of components is then $\sum_{t=1}^n 1/(2t-1) \sim (1/2) \log n$ and, since events are independent, we derive that a.a.s.\ there are $(1/2+o(1)) \log n$ components in $G_1^n$ by Chernoff's bound. In particular, $G_1^n$ is a.a.s.\ disconnected and thus contains no Hamiltonian cycle. A similar argument shows that a.a.s.\ many components of $G_1^n$ have an odd number of vertices, so $G_1^n$ has no perfect matching. In view of this, we will restrict our attention throughout the paper to the case $m\ge2$, for which it is known~\cite{BR} that $G_m^n$ is a.a.s.\ connected.

\smallskip

We finally consider the {\em uniform attachment} graph $G^n_{m-\old}$ on vertex set $[n]$, which can be thought of as an intermediate model between the $m$-out and preferential attachment models. In this setting, vertex $1$ has $m$ directed loops. Each vertex $v \in [n] \setminus \{1\}$ independently chooses $m$ random out-neighbours (with repetition) from $[v-1]$, to create the random digraph $D^n_{m-\old}$. We then obtain $G^n_{m-\old}$ by ignoring orientations. (Note the sole purpose of having loops in the definition is to ensure that the number of edges of $G^n_{m-\old}$ is precisely $mn$, but loops and multiple edges play no role in the existence of perfect matchings or Hamiltonian cycles and thus may be ignored.) We may also regard this model as a random process $(G^t_{m-\old})_{1\le t\le n}$, in which at each step a new vertex is created and attached only to older vertices (as in preferential attachment), but the choice of these vertices is uniform (as in the $m$-out model).
It turns out that the fact that edges are always generated towards older vertices causes one of the main difficulties in our analysis, which motivates a unified treatment for uniform and preferential attachment. In fact, we will often examine the uniform attachment case first for simplicity, and then adapt the argument to the non-uniform (rich-get-richer) distribution in preferential attachment. However, a more challenging and technical problem we encountered in transferring the result to not-uniform model was the fact that edges generated in the preferential attachment model are not generated independently. The two-round exposure technique (explained in Subsection~\ref{ssec:two-round}), used in the argument, has to be dealt with care. The lemmas needed to overcome all difficulties might be interesting on their own rights, and could be useful in future applications of this powerful technique.

\smallskip

In this paper, we prove that there exists a constant $c \in \nat$ such that a.a.s.\ both $G^n_m$ and $G^n_{m-\old}$ have a Hamiltonian cycle, provided that $m \ge c$. Of course, the existence of a Hamiltonian cycle implies the existence of a perfect matching\footnote{In the sequel, a perfect matching will denote a matching of size $\lfloor n/2\rfloor$ (thus allowing one unmatched vertex if $n$ is odd).}. However, we treat both properties independently in order to obtain smaller constants for the weaker property. Specifically, for perfect matchings we get $c=159$ for $G^n_{m-\old}$ (Theorem~\ref{thm:matchings-old}) and $c=1{,}260$ for $G^n_m$ (Corollary~\ref{cor:matchings}). For Hamiltonian cycles we get $c=3{,}214$ for $G^n_{m-\old}$ (Theorem~\ref{thm:cycles-old}) and $c=29{,}500$ for $G^n_m$ (Corollary~\ref{cor:cycles}). We tried to tune parameters in the argument to get the corresponding $c$'s as small as possible but clearly, with more effort, one may improve them further. However, it seems that with the existing argument it is impossible to find the threshold $c$ for any of the two properties and any of the two models. In particular, is it true that $G^n_{3}$ or $G^n_{3-\old}$ are a.a.s.\ Hamiltonian? This remains an open problem for now. On the other hand, all we managed to show is that a.a.s.\ $G^n_2$ has no perfect matching (and so also no Hamiltonian cycle) and that $G^n_{2-\old}$ has no Hamiltonian cycle (see discussion in Section~\ref{s:lower}).

\section{Preliminaries}

\subsection{Basic definitions and notation}\label{ssec:notation}

In order to state some of the intermediate results in the paper, it is convenient to extend the definitions of $G^n_{m-\old}$ and $G^n_{m}$ given in Section~\ref{sec:intro} to the case $m=0$ by interpreting both $G^n_{0-\old}$ and $G^n_{0}$ as the graph on vertex set $[n]$ and no edges. Fix any constant $m\in\ent_{\ge0}$. For any vertices $v,w\in[n]$ satisfying $w<v$, we say that $w$ is {\em older} than $v$ (or $v$ is {\em younger} than $w$), since $w$ is created earlier than $v$ in the process leading to $G^n_{m-\old}$ or $G^n_{m}$.
We say that an edge $e=vw$ \emph{stems} from $v$ if $w\le v$ (or equivalently $e$ is one of the $m$ edges that attach $v$ to younger vertices or to $v$ itself at the step of the process when $v$ is added to the graph).

Given a graph $G=(V,E)$ and a set $C \subseteq V$, we will use $N(C)$ to denote the \emph{neighbourhood} of $C$; that is,
\[
N(C) = \{ w \in V \setminus C :  vw \in E \text{ for some } v \in C \},
\]
and we write $N(v)=N(\{v\})$ for simplicity. All asymptotics throughout are as $n \rightarrow \infty $ (we emphasize that the notations $o(\cdot)$ and $O(\cdot)$ refer to functions of $n$, not necessarily positive, whose growth is bounded). We also use the notations $f \ll g$ for $f=o(g)$ and $f \gg g$ for $g=o(f)$. For simplicity, we will write $f(n) \sim g(n)$ if $f(n)/g(n) \to 1$ as $n \to \infty$ (that is, when $f(n) = (1+o(1)) g(n)$).  Since we aim for results that hold a.a.s.\ (see the definition in the previous section), we will always assume that $n$ is large enough. We often write, for example, $G^n_m$ when we mean a graph drawn from the distribution $G^n_m$. Given $k\in\ent_{\ge0}$, we define $k!!=k(k-2)(k-4)\cdots3\cdot1$ for odd $k$ and $k!!=k(k-2)(k-4)\cdots4\cdot2$ for even $k$. Finally, for any $x \in \real_+$, we use $\log x$ to denote the natural logarithm of $x$, and $[k]=\{1,2,\ldots,k\}$ for any $k \in \nat$.

\subsection{Useful bounds}

Most of the time, we will use the following version of \emph{Chernoff's bound}. Suppose that $X \in \Bin(n,p)$ is a binomial random variable with expectation $\mu=np$. If $0<\delta<1$, then 
$$
\pr [X < (1-\delta)\mu] \le \exp \left( -\frac{\delta^2 \mu}{2} \right),
$$ 
and if $\delta > 0$,
\[\pr [ X > (1+\delta)\mu] \le \exp\left(-\frac{\delta^2 \mu}{2+\delta}\right).\]
However, at some point we will need the following, stronger, version: for any $t \ge 0$, we have
\begin{equation}\label{eq:strongChernoff1}
\pr [X \ge \mu + t] \le \exp \left( - \mu \varphi \left( \frac {t}{\mu} \right) \right),
\end{equation}
and for any $0 \le t \le \mu$, we have
\begin{equation}\label{eq:strongChernoff2}
\pr [X \le \mu - t] \le \exp \left( - \mu \varphi \left( \frac {-t}{\mu} \right) \right),
\end{equation}
where
\begin{equation}\label{eq:phidef}
\varphi(x) = (1+x)\log(1+x)-x.
\end{equation}
Moreover, let us mention that all of these bounds hold for the general case in which $X$ is a sum of $\Bernoulli(p_i)$ indicator random variables with (possibly) different $p_i$. These inequalities are well known and can be found, for example, in~\cite{JLR}.

\medskip

We will also use a standard martingale tool: the \emph{Hoeffding-Azuma inequality}.  As before, see for example~\cite{JLR} for more details. Let $X_0, X_1, \ldots$ be a martingale. Suppose that there exist $c_1, c_2, \ldots,c_n >0$ such that $|X_k - X_{k-1}| \le c_k$ for each $1\le k \le n$. Then, for every $x >0$,
\begin{equation}\label{eq:HA-inequality1}
\pr [ X_n > \ex X_n + x ] \le \exp \left(-\frac{x^2}{2\sum_{k=1}^n {c_k}^2}\right).
\end{equation}

The Hoeffding-Azuma inequality can be generalized to include random variables close to martingales. One of our proofs, proof of Lemma~\ref{lem:total_weight_specific}, will use the supermartingale method of Pittel et al.~\cite{Pittel99}, as described in~\cite[Corollary~4.1]{Wormald-DE}. Let $X_0, X_2, \ldots, X_n$ be a sequence of random variable. Suppose that there exist $c_1, c_2, \ldots,c_n >0$ and $b_1, b_2, \ldots,b_n >0$ such that 
$$
|X_k - X_{k-1}| \le c_k \quad \text{ and } \quad \ex [X_k - X_{k-1} | X_{k-1} ] \le b_k
$$ 
for each $1\le k \le n$. Then, for every $x >0$,
\begin{equation}\label{eq:HA-inequality2}
\pr \left[ \text{For some $t$ with $0 \le t \le n$: } X_t - X_0 > \sum_{k = 1}^{t} b_k + x \right] \le \exp \left(-\frac{x^2}{2\sum_{k=1}^n {c_k}^2}\right).
\end{equation}

\bigskip

Finally, we include the following auxiliary lemma that will simplify some calculations in Lemma~\ref{lem:total_weight}.
(We use the convention that the empty product is equal to $1$ and the empty sum is equal to $0$.)
\begin{lem}\label{lem:aux}
Given any integers $a=a(n)$ and $b=b(n)$ such that $0\le a \le b$, define 
\[
c_{a,b} = \prod_{i=a+1}^b \left(\frac{2i-1}{2i}\right).
%\quad\text{if $a<b$}
%\qquad\text{and}\qquad
%c_{a,b}=1
%\quad\text{if $a=b$.}
\]
Then, 
\[
c_{a,b} = e^{O(\eps)} \sqrt{a/b},
\]
where $\eps = 1/(a+1)$, and the constants involved in the $O()$ notation do not depend on $a$ or $b$.
Moreover
\[
\sum_{i=a+1}^b {c_{a,i}}^2 = e^{O(\eps)} a\log(b/a).
\]
\end{lem}
\begin{proof}
The case $a=b$ is trivial, so we may assume $a<b$. Then, 
\begin{align*}
c_{a,b} &= \prod_{i=a+1}^b \left(1-\frac{1}{2i}\right)
= \prod_{i=a+1}^b \exp\left(-\frac{1}{2i} + O(i^{-2})\right)
= \exp\left(- \sum_{i=a+1}^b \frac{1}{2i} + O(\eps)\right)
\\
&= \exp\left(- \frac{1}{2}\log (b/a) + O(\eps)\right)
= e^{O(\eps)} \sqrt{a/b},
\end{align*}
where we used the fact that the harmonic sum satisfies $\sum_{i=1}^n 1/i = \log n + \gamma + O(1/n)$. This proves the first claim of the lemma. The following calculation yields the second claim:
%\xc{in the last step of the next equation, i'm hidding under the rug that $\log(b/a)+O(\eps) = \Theta(\log(b/a))$, so in particular there is no cancellation of terms. This is easy to prove considering separately the cases $b-a\to\infty$ and $b-a$ constant. I would be happy to ommit all this, unless any of you wants to add more detail.}
\[
\sum_{i=a+1}^b {c_{a,i}}^2 = e^{O(\eps)} a \sum_{i=a+1}^b 1/i = e^{O(\eps)} a \log(b/a).\qedhere
\]
\end{proof}

\subsection{Two-round exposure}\label{ssec:two-round}

Fix constants $m,m_1,m_2\in\ent_{\ge0}$ such that $m=m_1+m_2$, and recall the definitions in Section~\ref{ssec:notation}. For each vertex $v\in[n]$ of either model $G^n_{m-\old}$ or $G^n_{m}$, we colour blue the first $m_1$ edges stemming from $v$ and colour red the $m_2$ remaining ones.
We denote the resulting edge-coloured versions of $G^n_{m-\old}$ and $G^n_{m}$ by $G^n_{m_1,m_2-\old}$ and $G^n_{m_1,m_2}$, respectively. (Note that these edge colourings are not proper in the usual graph-theoretical sense.)
Let $\pi_1\big(G^n_{m_1,m_2-\old}\big)$ be the graph on vertex set $[n]$ whose edges are precisely the blue edges of $G^n_{m_1,m_2-\old}$, and symmetrically let $\pi_2\big(G^n_{m_1,m_2-\old}\big)$ be the graph with the red edges of $G^n_{m_1,m_2-\old}$. Moreover, let $\pi\big(G^n_{m_1,m_2-\old}\big)$ be the union of graphs $\pi_1\big(G^n_{m_1,m_2-\old}\big)$ and $\pi_2\big(G^n_{m_1,m_2-\old}\big)$ --- or simply $G^n_{m_1,m_2-\old}$ --- after removing the colours from the edges. By construction, $\pi\big(G^n_{m_1,m_2-\old}\big)$  is distributed as $G^n_{m-\old}$. Finally, graphs $\pi_1\big(G^n_{m_1,m_2}\big)$, $\pi_2\big(G^n_{m_1,m_2}\big)$ and $\pi\big(G^n_{m_1,m_2}\big)$ are defined analogously by replacing $G^n_{m_1,m_2-\old}$ by $G^n_{m_1,m_2}$.

Observe that, for each $\sigma\in\{1,2\}$, $\pi_\sigma\big(G^n_{m_1,m_2-\old}\big)$ has the same distribution as $G^n_{m_\sigma-\old}$ (so we can identify $\pi_\sigma\big(G^n_{m_1,m_2-\old}\big)$ and $G^n_{m_\sigma-\old}$), and moreover the graphs $\pi_1\big(G^n_{m_1,m_2-\old}\big)$ and $\pi_2\big(G^n_{m_1,m_2-\old}\big)$ are independent of each other.
In other words, the union of two independent uniform attachment graphs $G^n_{m_1-\old}$ and $G^n_{m_2-\old}$ is distributed as $G^n_{m-\old}$ with $m=m_1+m_2$. This allows us to build $G^n_{m-\old}$ by first exposing the edges of $G^n_{m_1-\old}$ (blue edges) and then adding the edges of $G^n_{m_2-\old}$ (red edges).
This technique is know as the two-round exposure method, and will be repeatedly used in Section~\ref{sec:upper} to build a perfect matching and a Hamiltonian cycle in $G^n_{m-\old}$. 

Unfortunately, the analogous property does not hold for the preferential attachment model (unless\footnote{If $m_2=0$ then $\pi_2\left(G^n_{m_1,m_2}\right)$ has no edges, so $\pi_1\left(G^n_{m_1,m_2}\right) = G^n_{m} = G^n_{m_1}$, and vice-versa.} $m_1$ or $m_2$ equals $0$).
That is, if $m_1,m_2>0$, then $G^n_{m-\old}$ is not distributed as the union of independent $G^n_{m_1-\old}$ and $G^n_{m_2-\old}$.
This is due to the fact that, when a new edge stemming from a vertex $v$ is created, the choice of the second endpoint depends on all the edges (both blue and red) previously added during the process.
We want to use a version of the two-round exposure technique, and build the preferential attachment graph $G^n_{m}$ as the union of $\pi_1\left(G^n_{m_1,m_2}\right)$ and $\pi_2\left(G^n_{m_1,m_2}\right)$. However (again if $m_1,m_2>0$),  for each $\sigma\in\{1,2\}$, graph $\pi_\sigma\left(G^n_{m_1,m_2}\right)$ is not distributed as $G^n_{m_\sigma}$, and moreover $\pi_1\left(G^n_{m_1,m_2}\right)$ and $\pi_2\left(G^n_{m_1,m_2}\right)$ are not independent of each other.
In spite of all these obstacles, our argument will analyze $\pi_1\left(G^n_{m_1,m_2}\right)$ and $\pi_2\left(G^n_{m_1,m_2}\right)$ separately, in a similar fashion as we deal with $G^n_{m_1-\old}$ and $G^n_{m_2-\old}$ for the uniform attachment model.
This time though, we will need to use properties of $\pi_1\left(G^n_{m_1,m_2}\right)$ that hold a.a.s.\ conditional on a given $\pi_2\left(G^n_{m_1,m_2}\right)$ and vice-versa.

\subsection{Relationship between $G^n_m$ and $G^n_{m-\old}$}\label{ssec:relationship}

In this section, we will analyze some basic features of $G^n_{m-\old}$ and $\pi_\sigma\big(G^n_{m_1,m_2}\big)$, where $\sigma\in\{1,2\}$ and $m,m_1,m_2$ are fixed nonnegative integers such that $m=m_1+m_2$. As already mentioned, if $m_2=0$ then $\pi_1\big(G^n_{m_1,m_2}\big)$ is simply the preferential attachment model $G^n_{m}$. Hence, various lemmas stated below for $G^n_{m_1,m_2}$ also apply to $G^n_{m}$.

Given any vertex $v\in[n]$ and any set $W\subseteq[v-1]$ of vertices older than $v$, the probability in $G^n_{m-\old}$ that there is no edge between $v$ and $W$ is precisely
\begin{equation}\label{eq:m-old}
\left( 1 - \frac{|W|}{v-1} \right)^m.
\end{equation}
Moreover, the corresponding event is independent from all edges added during the construction of $G^n_{m-\old}$ stemming from vertices different from $v$.
This is a crucial fact in our argument, which unfortunately does not extend to the preferential attachment model $G^n_m$, since the attachments of vertex $v$ are not independent from ``the future" (i.e.~they are not independent from the edges that stem from vertices younger than $v$, which are added later in the process).
Fortunately,  we can obtain a weaker statement that will suffice for our purposes, and may be useful for other applications.

\begin{lem}\label{lem:bound_no_edges_future}
Fix any constants $m,m_1,m_2\in\ent_{\ge0}$ with $m=m_1+m_2$, and let $\sigma\in\{1,2\}$. Given any vertex $v\in[n]$ and any set $W\subseteq[v-1]$, let $E$ be any event of $G^n_{m_1,m_2}$ that does not involve any edges stemming from $v$ in $\pi_\sigma\big(G^n_{m_1,m_2}\big)$.
Then, conditional upon $E$, the probability that $\pi_\sigma\big(G^n_{m_1,m_2}\big)$ has no edges between $v$ and $W$ is at most
\[
\left( 1 - \frac {|W|}{v+n} \right)^{m_\sigma} \le \left( 1 - \frac {|W|}{2n} \right)^{m_\sigma}.
\]
\end{lem}
\begin{proof}
The statement is trivially true if $m_\sigma=0$, so we assume that $m_\sigma\ge1$.

Let $e_1,\ldots,e_{m_\sigma}$ be the $m_\sigma$ edges that stem from $v$ in $\pi_\sigma(G^n_{m_1,m_2})$. Recall that these edges are blue if $\sigma=1$, and red if $\sigma=2$.
For each $i\in[m_\sigma]$ and each $u\in[v]$, let $B_{i,u}$ be the event that $e_i$ is attached to vertex $u$, and let $B_i=\bigcup_{u\in W} B_{i,u}$ be the event that $e_i$ is attached to some vertex in $W$.
Also, let $E'_i$ be any event that involves only edges (blue or red) of $G^n_{m_1,m_2}$ that are different from $e_i$.
\begin{claim}\label{singleclaim}
Fix $i\in[m_\sigma]$. The probability of $B_i$ conditional upon $E'_i$ is at least $|W|/(v+n)$.
\end{claim}
The lemma follows immediately by repeatedly applying Claim~\ref{singleclaim} to each $B_i$ ($i\in[m_\sigma]$) with
\[
E'_i = E\cap\overline{B_1}\cap\overline{B_2}\cap\cdots\cap\overline{B_{i-1}}.
\]
We proceed to prove Claim~\ref{singleclaim}. Since $i\in[m_\sigma]$ is fixed, we will omit for simplicity the subindex $i$ from the notation. The colours of the edges play no role in this claim, so we may ignore them and regard $G^n_{m_1,m_2}$ simply as $G^n_m$ with $m=m_1+m_2$. Also, we can assume that event $E'$ fully determines all $mn-1$ edges of the process $\left(G^t_m\right)_{1\le t\le n}$ other than $e$. (Indeed, any other situation can be trivially reduced to this one by using the law of total probabilities.) Moreover, we may assume that $E'$ not only determines edges that are introduced during the process but also an order in which they were added.

For each $u\in[n]\setminus \{v\}$, let $d_u$ be the degree of $u$ in $G^n_m - e$ conditional upon $E'$ (each self-loop of $u$ contributes with $2$ to this degree). Define $d_v$ analogously but adding an extra $1$ to account for the contribution of the first endpoint of $e$.
It is easy to observe from the definition of the process $\left(G^t_m\right)_{1\le t\le n}$ that, given any $u\in[v-1]$, 
\begin{equation}\label{eq:PBuE}
\pr(B_u \cap E') = \frac{d_u\prod_{w\in[n]} f_w}{(mn-1)!!},
\end{equation}
where each factor $f_w$ depends only on $E'$. For instance, if vertex $w\in[n]$ has no self-loops in $G^n_m - e$ given $E'$ then $f_w=(d_w-1)!/(m-1)!$. Similar expressions for $f_w$ can be easily computed in the case that $w$ has some self-loops (the value of $f_w$ depends on which edges stemming from $w$ create a loop). The case $u=v$ is slightly different, but we have
\begin{equation}\label{eq:PBvE}
\pr(B_v \cap E') \le \frac{d_v\prod_{w\in[n]} f_w}{(mn-1)!!}.
\end{equation}
%
%\xcd{We should perhaps be more explicit defining these $f_w$ in general and also give an exact expression for~\eqref{eq:PBvE}. I can do that later if we are all convinced of the correctness of the proof.} \pc{As you wish. I thought it is clear what is going on and I did not care about explicit definition. But if you really want, then go ahead.}
In view of~\eqref{eq:PBuE} and~\eqref{eq:PBvE}, for every $u\in[v-1]$,
%\[
%\frac{\pr(B_u \cap E)}{\pr(B_w \cap E)} = \frac{d_u+f_u+1}{d_w+f_w+1}
%\]
\[
\pr(B_u \mid E') =
\frac{\pr(B_u \cap E')}{\sum_{w\in[v]}\pr(B_w \cap E')} \ge
\frac{d_u}{\sum_{w\in[v]} d_w} \ge
\frac{d_u}{2mv + m(n-v) } =
\frac{d_u}{m(v+n)}.
\]
Therefore, summing over $u\in W$,
\[
\pr(B \mid E') \ge \sum_{u\in W} \frac{d_u}{m(v+n)} \ge \frac{m|W|}{m(v+n)} = \frac{|W|}{v+n},
\]
as required. This yields Claim~\ref{singleclaim} and finishes the proof of the lemma.
\end{proof}
Note that the bound in the statement of Lemma~\ref{lem:bound_no_edges_future} could be much worse than~\eqref{eq:m-old}, especially if $v$ is much smaller than $n$. However, in those cases in which event $E$ does not depend upon the ``future" of vertex $v$ (i.e.~does not depend on edges stemming from younger vertices in the process), we can derive a much better bound which resembles that of~\eqref{eq:m-old}.

\begin{lem}\label{lem:bound_no_edges}
Fix any constants $m,m_1,m_2\in\ent_{\ge0}$ with $m=m_1+m_2$ and let $\sigma\in\{1,2\}$. Given any vertex $v\in[n]$ and any set $W\subseteq[v-1]$,
let $E$ be any event that involves only edges that stem from vertices in $[v-1]$ (i.e.~older than $v$) in the process leading to $G^n_{m_1,m_2}$. Then, conditional upon $E$,
the probability that $\pi_\sigma\big(G^n_{m_1,m_2}\big)$ has no edges between $v$ and $W\subseteq[v-1]$ is at most
%Fix any constants $m,m_1,m_2\in\ent_{\ge0}$ with $m=m_1+m_2$ and let $\sigma\in\{1,2\}$. Given any vertex $v\in[n]$ and any set $W\subseteq[v-1]$,
%the probability that $\pi_\sigma\big(G^n_{m_1,m_2}\big)$ has no edges between $v$ and $W\subseteq[v-1]$ is at most
%
\[
\left( 1 - \frac {|W|}{2v} \right)^{m_\sigma}.
\]
%
%Moreover, this probability bound is valid conditional on any event that involves only edges that stem from vertices in $[v-1]$ in the process leading to $G^n_{m_1,m_2}$.
\end{lem}
\begin{proof}
Suppose that all (blue and red) edges of $G^n_{m_1,m_2}$ stemming from vertices in $[v-1]$ have already been exposed. Additionally, suppose that some edges  (or perhaps none) stemming from $v$ have been exposed as well. We add a new edge $e$ stemming from $v$. At that time, the total degree of $W$ (i.e.~the sum of the degrees of all its vertices) is at least $m|W|$, and the the total degree of the graph (counting $e$ as already contributing one to the degree of $v$) is at most $2m v$. Therefore, the conditional probability that $e$ does not join $v$ to a vertex in $W$ is at most
\[
1 - \frac {m|W|}{2mv} = 1 - \frac {|W|}{2v}.
\]
The lemma follows immediately by considering the $m_\sigma$ edges stemming from $v$ in $\pi_\sigma\big(G^n_{m_1,m_2}\big)$.
\end{proof}

In view of this, some of the proofs for $G^n_{m-\old}$ in the paper that use~\eqref{eq:m-old} can be easily adapted to $\pi_\sigma\big(G^n_{m_1,m_2}\big)$ (and thus to $G^n_m$) by adjusting constants. Most of the time, Lemma~\ref{lem:bound_no_edges_future} or Lemma~\ref{lem:bound_no_edges} will be enough to do it.
However, at some point, we will need an upper bound on the probability in $\pi_\sigma\big(G^n_{m_1,m_2}\big)$ that all the neighbours of $v$ fall inside of $W$. In order to do that we will use Lemma~\ref{lem:extra} below but, before we can state it, we need the following lemma that bounds the total degree of a set $W$ at the time vertex $v$ is created in the process leading to $G^n_{m}$.

\begin{lem}\label{lem:total_weight} 
Let $A>0$ be a sufficiently large constant. Fix any constant $m\in\ent_{\ge0}$, and let $\omega=\omega(n)$ be any function tending to infinity as $n \to \infty$. The following property holds a.a.s.\ for $G^n_m$. For any $1 \le k \le t \le n$ and any set $W \subseteq [t]$ of size $k$,
$$
\sum_{w \in W} \deg(w,t) \le 
\begin{cases}
\big(1+A/(mk+1)\big) \left( 2m \sqrt{kt}  + \sqrt{8mkt}\log(et/k) \right)
& \text{if $t \ge \omega$,} \\
2m \omega & \text{otherwise.}
\end{cases}
$$
\end{lem}
\begin{proof}
The statement is trivially true if $m=0$ since $G^n_0$ has no edges, so assume that $m\ge1$. Moreover, the case $k=t$ (which corresponds to $W=[t]$) follows immediately from the case $W=[t-1]$ (at the only expense of replacing $A$ by a slightly larger absolute constant), so we will ignore it.

Fix $k$ and $t$ such that $1\le k < t \le n$, and assume $t\ge\omega\to\infty$ as $n \to \infty$. We will need to understand the behaviour of the following random variable:
$$
Y_t (k) = \sum_{j=1}^{k} \deg(j, t).
$$
(In other words, $Y_t(k)$ is the sum of the degrees of the $k$ oldest vertices (at time $t$).)
In view of the identification between the models $G_m^n$ and $G_1^{mn}$, it will be useful to investigate the following collection of random variables: for $m k \le s \le mn$, let
$$
X_s = \sum_{j'=1}^{mk} \deg_{G_1^{t}}(j', s).
$$

Clearly, $Y_t (k) = X_{mt}$ and $X_{m k} = 2m k$. Moreover, for $mk < s \le mn$,
$$
X_s  =
\begin{cases}
X_{s-1}+1 & \text{with probability } \frac {X_{s-1}}{2s-1}, \\
X_{s-1} & \text{otherwise}.
\end{cases}
$$
The conditional expectation is given by
$$
\ex \big( X_s | X_{s-1} \big) = (X_{s-1}+1) \; \frac {X_{s-1}}{2s-1} + X_{s-1} \left(1- \frac {X_{s-1}}{2s-1} \right) = \frac {2s}{2s-1} \; X_{s-1}.
$$
Define $\widehat X_{mk} = X_{mk} = 2mk$ and
\[
\widehat X_s = \left( \prod_{i=mk+1}^s \frac {2i-1}{2i} \right) X_s
\qquad
\text{for $mk < s \le mn$,}
\]
or equivalently $\widehat X_s = c_{mk,s}\, X_s$ with $c_{mk,s}$ defined as in Lemmma~\ref{lem:aux}. We have that $\ex \big( \widehat X_s | \widehat X_{s-1} \big) = \widehat X_{s-1}$, and thus $\widehat X_{mk},\ldots,\widehat X_{mn}$ is a martingale with $\ex \widehat X_s = 2mk$.
Moreover, the difference between consecutive terms is
\[
|\widehat X_s - \widehat X_{s-1}| = c_{mk,s} \left| X_s - \left(1+ \frac {1}{2s-1} \right) X_{s-1} \right|
= c_{mk,s} \left| X_s -  X_{s-1} - \frac {X_{s-1}}{2s-1} \right| \le c_{mk,s},
\]
since $X_s -  X_{s-1} \in \{0,1\}$ and $0<\frac {X_{s-1}}{2s-1} \le \frac {mk+s-1}{2s-1} < 1$.
Also, in view of Lemma~\ref{lem:aux}, there exist absolute constants $A_1,A_2>0$ such that 
\begin{equation}\label{eq:cmkmt}
c_{mk,mt} \ge e^{-A_1/(mk+1)}\sqrt{k/t}
\qquad\text{and}\qquad
\sqrt{ 2 \sum_{i=mk+1}^{mt} {c_{mk,i}}^2 } \le e^{A_2/(mk+1)}\sqrt{2mk\log(t/k)}.
\end{equation}
Assume that constant $A$ in the statement is sufficiently large so that
$1+A/x \ge e^{(A_1+A_2)/x}$ for all $x\ge1$ (in particular, $A = e^{A_1+A_2}-1$ works), and define
\begin{align}
D &:= \big(1+A/(mk+1)\big) \left( 2m \sqrt{kt}  + \sqrt{8mkt}\log(et/k) \right)
\notag\\
& \ge e^{(A_1+A_2)/(mk+1)} \left( 2m \sqrt{kt}  + \sqrt{8mkt}\log(et/k) \right).
\label{eq:D}
\end{align}
We wish to bound the probability that $Y_t(k) > D$.
Observe that this event implies that 
\begin{align*}
&\frac{\widehat X_{mt} - 2mk}{\sqrt{2\sum_{i=mk+1}^{mt} {c_{mk,i}}^2}} = \frac{c_{mk,mt} Y_t(k) - 2mk}{\sqrt{2\sum_{i=mk+1}^{mt} {c_{mk,i}}^2}}
> \frac{c_{mk,mt} D - 2mk}{\sqrt{2\sum_{i=mk+1}^{mt} {c_{mk,i}}^2}} \ge
\\
& \ge \frac{ e^{A_2/(mk+1)} \left( 2m k  + \sqrt{8m}k\log(et/k) \right) - 2mk}{e^{A_2/(mk+1)}\sqrt{2mk\log(t/k)}} >
\\
& > \frac{ 2\sqrt{k}\log(et/k) }{ \sqrt{\log(t/k)}}
> 2\sqrt{\log \big( (et/k)^k \big)}
\ge \sqrt{\log \left( \binom{t}{k}^4 \right)}
\ge \sqrt{\log \left(t^3\binom{t}{k}\right)},
\end{align*}
where we used~\eqref{eq:cmkmt}, \eqref{eq:D} and the facts that $1\le k<t$ and $t \le \binom{t}{k}\le (et/k)^k$. Hence, applying Hoeffding-Azuma inequality~(\ref{eq:HA-inequality1}) to the martingale $\widehat X_{mk},\ldots,\widehat X_{mn}$ we get that
\[
\pr \left( Y_t(k) > D \right) \le
\pr \left( \frac{\widehat X_{mt} - 2mk}{\sqrt{2\sum_{i=mk+1}^{mt} {c_{mk,i}}^2}}  > \sqrt{\log \left(t^3\binom{t}{k}\right)} \right) \le e^{- \log \left(t^3\binom{t}{k}\right)} = 1 \bigg/ t^{3} \binom{t}{k}.
\]
Observe that for any set of vertices $W$ of size $k$, each random variable $\sum_{j'\in W} \deg_{G_1^{t}}(j', s)$ is stochastically dominated by $X_s$ (for $mk\le s\le mn$). Therefore,
\[
\pr\left (\sum_{w \in W} \deg(w,t) > D \right) \le \pr \left( Y_t(k) > D \right) \le 1 \bigg/ t^{3} \binom{t}{k}.
\]
Since there are at most ${t \choose k}$ sets $W \subseteq [t]$ of size $k$ to consider, the desired property fails for a given $t$ with probability at most $\sum_{k=1}^{t-1} t^{-3} \le  t^{-2}$. Hence, a.a.s.\ it does not fail for any $\omega \le t \le n$. As $Y_{\omega}(\omega) = 2m\omega$, the desired bound trivially holds (deterministically) for $t \le \omega$, and the proof is finished.
\end{proof}

We will also need a stronger result for sets of a certain type.

\begin{lem}\label{lem:total_weight_specific} 
Fix any constant $c\in (0,1)$ and $m\in\ent_{\ge0}$. The following property holds a.a.s.\ for $G^n_m$. For any $cn \le t \le n$, 
$$
Y_t := \sum_{w \in [cn]} \deg(w,t) \sim 2mn \sqrt{ct/n}.
$$
\end{lem}
\begin{proof}
In view of the identification between the models $G_m^n$ (on the vertex set $1, 2, \ldots, n$) and $G_1^{mn}$ (on the vertex set $1', 2', \ldots, mn'$), it will be useful to investigate the following random variable instead of $Y_t$: for $m\lfloor cn\rfloor \le t \le mn$, let
$$
X_t = \sum_{j \in [cmn]} \deg_{G_1^{t}}(j', t).
$$
Clearly, $Y_t = X_{tm}$. It follows that $X_{m\lfloor cn\rfloor} = Y_{\lfloor cn\rfloor} = 2m\lfloor cn\rfloor$. Moreover, for $m\lfloor cn\rfloor < t \le mn$,
$$
X_t =
\begin{cases}
X_{t-1}+1 & \text{with probability } \frac {X_{t-1}}{2t-1}, \\
X_{t-1} & \text{otherwise}.
\end{cases}
$$
The conditional expectation is given by
$$
\ex \left( X_t | X_{t-1} \right) = (X_{t-1}+1) \cdot \frac {X_{t-1}}{2t-1} + X_{t-1} \left(1- \frac {X_{t-1}}{2t-1} \right) = X_{t-1} \left(1+ \frac {1}{2t-1} \right).
$$
Taking expectation again, we derive that
$$
\ex {X_t} = \ex { X_{t-1} } \left(1+ \frac {1}{2t-1} \right).
$$
Hence, it follows that
$$
\ex (Y_t) = \ex ( X_{tm})  = 2m\lfloor cn\rfloor \prod_{s=m\lfloor cn\rfloor+1}^{tm}  \left(1+ \frac {1}{2s-1} \right) \sim 2cmn \left( \frac {tm}{cmn} \right)^{1/2} = 2mn \sqrt{ct/n}.
$$

%Noting that $\ex {Y_t} = \Theta(n)$ for any $\lfloor cn\rfloor \le t \le n$, and that $Y_t$ increases by at most $m$ each time ($X_t$ increases by at most one), we obtain that with probability $1-o(n^{-1})$, $Y_t = \ex {Y_t} + O(\sqrt{n \log n}) \sim \ex {Y_t}$ (using a standard martingale argument; for Azuma-Hoeffding inequality that is used here see, for example, \cite{JLR}). Hence, we may assume that $Y_t \sim 2mn \sqrt{ct/n}$ for any $\lfloor cn\rfloor \le t \le n$.

In order to transform $X_t$ into something close to a martingale (to be able to apply the generalized Azuma-Hoeffding inequality~(\ref{eq:HA-inequality2})), we set for $m \lfloor cn \rfloor \le t \le mn$
$$
Z_t = X_t - 2m \lfloor cn \rfloor - \sum_{k=m \lfloor cn \rfloor + 1}^{t} \sqrt{cmn/k}
$$ 
(note that $Z_{m \lfloor cn \rfloor} = 0$) and use the following stopping time
$$
T = \min \left\{ t > m \lfloor cn \rfloor : X_t \le 2 \sqrt{t cmn} + t^{2/3} \text{ or } t = mn \right\}.
$$
Indeed, we have for $m \lfloor cn \rfloor < t \le mn$
$$
\ex \left( Z_t - Z_{t-1} ~~|~~ Z_{t-1} \right) = \frac {X_{t-1}}{2t-1} - \sqrt{cmn/t} \le (1/2+o(1)) t^{-1/3} < 0.51 t^{-1/3},
$$
provided $t \le T$, and $|Z_t - Z_{t-1}| \le 1$ as $t > cmn$.
Let $t \wedge T$ denote $\min\{t,T\}$. We apply the generalized Azuma-Hoeffding inequality~(\ref{eq:HA-inequality2}) to the sequence $(Z_{t \wedge T} : m \lfloor cn \rfloor \le t \le mn)$, with $c_t = 1$, $b_t = 0.51 t^{-1/3}$ and $x = 0.1 t^{2/3}$, to conclude that a.a.s.\ for all $t$ such that $m \lfloor cn \rfloor \le t \le mn$
$$
Z_{t \wedge T} - Z_{m \lfloor cm \rfloor} = Z_{t \wedge T} \le  \sum_{k\le t} b_k + x \le 0.77 t^{2/3} + 0.1 t^{2/3} \le 0.9 t^{2/3}.
$$

To complete the proof we need to show that a.a.s., $T = mn$. The events asserted by the equation hold a.a.s.\ up until time $T$, as shown above. Thus, in particular, a.a.s.
\begin{eqnarray*}
X_T &=& Z_T + 2m\lfloor cn \rfloor + \sum_{k=m \lfloor cn \rfloor + 1}^{t} \sqrt{cmn/k} \\
&\le& 0.9 t^{2/3} + 2mcn + \sqrt{cmn} \int_{mcn}^{t} 1/\sqrt{k} \; dk + O(1) \\
&<& 2\sqrt{tcmn} + t^{2/3},
\end{eqnarray*}
which implies that $T = mn$ a.a.s. In particular, it follows that a.a.s.,for any $cn \le t \le n$, $Y_{t} = X_{mt} < 2 mn \sqrt{ct/n} + o(n)$. The lower bound can be obtained by applying the same argument symmetrically to $(-Z_{t \wedge T} : m \lfloor cn \rfloor \le t \le mn)$, and so the proof is finished.
\end{proof}

Finally, we provide the last tool we need in order to extend the arguments in the paper from $G^n_{m-\old}$ to $\pi_\sigma\big(G^n_{m_1,m_2}\big)$ (and thus to $G^n_m$ if $m_1$ or $m_2$ equals $0$).
\begin{lem}\label{lem:extra}
Let $C>0$ be a sufficiently large constant.
Fix any constants $m,m_1,m_2\in\ent_{\ge0}$ with $m=m_1+m_2 \ge 1$ and let $\sigma\in\{1,2\}$.
Let $\omega=\omega(n)$ be any function tending to infinity as $n \to \infty$, and suppose that the a.a.s.\ events in Lemma~\ref{lem:total_weight} hold in
$G^n_m = \pi\big(G^n_{m_1,m_2}\big)$. Then,
for any $\omega\le j\le n$, $R \subseteq [n]\setminus[j]$ and $\emptyset \ne Q\subseteq[n]$, the conditional probability that every vertex $v\in R$ has all neighbours in $Q$ with respect to graph $\pi_\sigma\big(G^n_{m_1,m_2}\big)$ is at most
\[
\left( \left(1+\tfrac{C}{|Q|}+o(1) \right) \left( 2 \sqrt{\tfrac{|Q|}{j}}  + \sqrt{\tfrac{8|Q|}{mj}}\log\left(e\left(1+\tfrac{j}{|Q|}\right)\right)  \right) \right)^{m_\sigma r}.
\]
\end{lem}
\begin{proof}
Put $|R|=r$ and $|Q|=q\ge 1$. Since the lemma is trivial for $m_\sigma r=0$, we will assume that $m_\sigma\ge1$ and $r\ge1$.
For $t,m,x\ge1$, define
\[
f_{t,m}(x) = 2m \sqrt{xt}  + \sqrt{8mxt}\log\left(e\left(1+t/x\right)\right).
\]
We will use the following observation that we will be proved at the very end.
\begin{claim}\label{c:f}
Given any constant $A>0$, there exists $A'>0$ sufficiently large such that for every $t,m\ge1$ and every $1\le x\le y$, $(1+A/x)f_{t,m}(x) \le (1+A'/y)f_{t,m}(y)$.
\end{claim}
Let $H$ denote the a.a.s.\ event in the statement of Lemma~\ref{lem:total_weight}, and let $A>0$ be the constant in that same lemma. 
For each $t\in[n]$, define $W_t = Q\cap[t]$, so in particular we have $|W_t|\le q$. In view of Claim~\ref{c:f}, event $H$ implies that, for every $t\in\nat$ such that $j\le t\le n$ and $W_t\ne\emptyset$, 
\begin{align*}
\sum_{w\in W_t} \deg(w,t)
&\le \left(1+\frac{A}{m|W_t|+1}\right) \left( 2m \sqrt{|W_t|t}  + \sqrt{8m|W_t|t}\log\left(et/|W_t|\right) \right)
\\
&\le \left(1+\frac{A}{|W_t|}\right) f_{t,m}(|W_t|)
\\
&\le \left(1+\frac{A'}{q}\right) f_{t,m}(q).
\end{align*}
(The last conclusion of the equation above is also true if $W_t=\emptyset$, so we can ignore the restriction we had on $W_t$. Moreover, let us mention that $|W_t| \le t$ but it is not always true that $q \le t$. This technical reason prevented us from defining  function $f_{t,m}(x)$ as follows: $f_{t,m}(x) = 2m \sqrt{xt}  + \sqrt{8mxt}\log\left(et/x\right).$)
For each $t\in[n]$ such that $t\ge j+1$, let $H_t$ be the event in $G^n_m = \pi\big(G^n_{m_1,m_2}\big)$
that, for every $j\le s\le t-1$, we have
\begin{equation*}%\label{eq:Ht}
\sum_{w\in W_s} \deg(w,s)\le \left(1+\frac{A'}{q}\right) f_{s,m}(q).
\end{equation*}
By construction, $H_{j+1} \supseteq H_{j+2} \supseteq \cdots H_n \supseteq H$.
For each $t\in[n]$, let $E_t$ be the event that every edge of $\pi_\sigma\big(G^n_{m_1,m_2}\big)$ stemming from vertex $t$ attaches this vertex to some vertex  in $Q$. Our goal is to bound $\pr \left( \bigcap_{t\in R}E_t \cap H \right)$. By labelling the $r$ vertices in $R$ as $t_1<t_2<\cdots<t_r$ and since $H_{j+1} \supseteq H_{j+2} \supseteq \cdots H_n \supseteq H$, we conclude that
\begin{equation}\label{eq:PEH}
\pr \left( \bigcap_{t\in R}E_t \cap H \right) \le
\pr \left( \bigcap_{i=1}^r \left( E_{t_i} \cap  H_{t_{i+1}} \right) \right) \le
\prod_{i=1}^r \pr \left( E_{t_i}  \;\big|\; \widehat H_{t_i} \right),
\end{equation}
where $\widehat H_{t_i} := H_{t_i} \cap \bigcap_{k=1}^{i-1} E_{t_k}$. Observe that event $\widehat H_{t_i}$
implies that
\begin{align*}
\sum_{w\in Q \cap[t_i-1]} \deg(w,t_i-1) \le \left(1+\frac{A'}{q}\right) f_{t_i-1,m}(q),
\end{align*}
and (crucially!) only exposes information concerning edges created before time $t_i$, so the probability of $E_{t_i}$ conditional upon $\widehat H_{t_i}$ is at most
\begin{align}
\left( \frac{\left(1+\frac{A'}{q}\right) f_{t_i-1,m}(q) + 2m}{m(t_i-1)}\right)^{m_\sigma}
&\le \left( \left(1+\tfrac{A'}{q}\right) \left( 2 \sqrt{\tfrac{q}{t_i-1}}  + \sqrt{\tfrac{8q}{m(t_i-1)}}\log\left(e\left(1+\tfrac{t_i-1}{q}\right)\right)  \right) + \tfrac{2}{t_i-1}\right)^{m_\sigma}
\notag\\
&\le \left( \left(1+\tfrac{A'}{q}\right) \left( 2 \sqrt{\tfrac{q}{j}}  + \sqrt{\tfrac{8q}{mj}}\log\left(e\left(1+\tfrac{j}{q}\right)\right)  \right) + \tfrac{2}{j}\right)^{m_\sigma},
\label{eq:PEH2}
\end{align}
where we used that $t_i-1\ge j$ (since $t_i\in R$) and the fact that $\frac{\log(e(1+x))}{\sqrt x}$ is decreasing with respect to $x$ in $(0,+\infty)$ (which follows from elementary analysis). Hence, setting $C=A'+1$, we obtain
\begin{equation}\label{eq:PEH3}
\pr \left( E_{t_i}  \;\big|\; \widehat H_{t_i} \right)
\le \left( \left(1+\tfrac{C}{q}\right) \left( 2 \sqrt{\tfrac{q}{j}}  + \sqrt{\tfrac{8q}{mj}}\log\left(e\left(1+\tfrac{j}{q}\right)\right)  \right) \right)^{m_\sigma},
\end{equation}
where the case $q\le j$ follows immediately from~\eqref{eq:PEH2} (since $2/j \le \frac{1}{q} 2\sqrt{q/j}$), and the case $q\ge j$ is trivially true (since the right-hand side of~\eqref{eq:PEH3} is greater than~1).
Combining~\eqref{eq:PEH} and~\eqref{eq:PEH3} together and using the fact that $H$ holds a.a.s., we conclude that
\begin{align*}
\pr \left( \bigcap_{t\in R}E_t \;\Big|\; H \right)
&= (1+o(1)) \pr \left( \bigcap_{t\in R}E_t \cap H \right)
\\
&\le \left( \left(1+\tfrac{C}{q}+o(1) \right) \left( 2 \sqrt{\tfrac{q}{j}}  + \sqrt{\tfrac{8q}{mj}}\log\left(e\left(1+\tfrac{j}{q}\right)\right)  \right) \right)^{m_\sigma r},
\end{align*}
which yields the statement of the lemma under the assumption that Claim~\ref{c:f} is valid.

Finally, we proceed to prove Claim~\ref{c:f}. We will only sketch the main steps in the argument and leave the details to the reader. We may increase $A$ if needed and assume it is a sufficiently large constant (independent of $t$ and $m$). Take $A'=A^5$. For each $t,m,x\ge1$, define $g_{t,m}(x)=(1+A/x)f_{t,m}(x)$. Elementary (but rather tedious) computations show that
\[
\frac{d}{dx} f_{t,m}(x)\ge0
\quad \text{for all $x\ge1$}
\quad\qquad\text{and}\qquad\quad
\frac{d}{dx} g_{t,m}(x)\ge0
\quad \text{for all $x\ge 7A \ge \tfrac{\sqrt{m/2}+1+3\frac{t}{t+x}}{\sqrt{m/2}+1-\frac{t}{t+x}} A$.}
\]
Fix any $t,m\ge1$ and any $1\le x\le y$. If $x\ge 7A$, then trivially
\[
g_{t,m}(x) \le g_{t,m}(y) \le (1+A'/y)f_{t,m}(y),
\]
as desired. Also, if $y\le A^4$,
\[
g_{t,m}(x) \le (1+A)f_{t,m}(x) \le (1+A)f_{t,m}(y) \le (1+A^5/y)f_{t,m}(y).
\]
Otherwise, suppose $x\le 7A$ and $y\ge A^4$. Then, assuming that $A$ is a large enough constant, we get
\[
g_{t,m}(x) \le (1+A)f_{t,m}(x) \le (1+A)f_{t,m}(7A) \le (1+A)\sqrt{7A}\left(2m\sqrt{t}+\sqrt{8mt}\log(et)\right)
\le f_{t,m}(A^4),
\]
where the last step follows from simple computations by considering separately the cases $t\le A^5$ and $t\ge A^5$. Hence,
\[
g_{t,m}(x) \le f_{t,m}(A^4) \le f_{t,m}(y) \le (1+A'/y)f_{t,m}(y),
\]
the proof of the claim is completed, and so the proof is finished.
\end{proof}

\section{Upper bound}\label{sec:upper}

\subsection{Expansion properties}\label{ssec:expansion}

Let us start with investigating some properties of $G^n_{m-\old}$ and $G^n_m$ that will turn out to be important in determining the existence of perfect matchings and Hamiltonian cycles.
All the results of this section are stated for $G^n_{m-\old}$ and $\pi_1\big(G^n_{2m,m'}\big)$ where $m\in\nat$ and $m'\in\ent_{\ge0}$, so they also apply to
$G^n_{2m}$ (which can be obtained from $\pi_1\big(G^n_{2m,m'}\big)$ by setting $m'=0$).
First we will need the following technical lemma.

\begin{lem}\label{lem:goodold}
Fix any constants $0<x,y,d,\alpha<1$, $m\in\nat$, and $m'\in\ent_{\ge0}$ satisfying
\[
(1-d)\varphi(-x)m-d>0
\qquad\text{and}\qquad
\alpha < \left( \frac {0.99y}{e} \right)^{ 1 / ((1-d)\phi(-x)m-d) },
\]
where $\varphi$ is defined in~\eqref{eq:phidef}.
Then, the following holds a.a.s.\ for $G^n_{m-\old}$ and for $\pi_1 (G^n_{2m,m'})$.
For any integer $k$ with $1\le k\le \alpha n$ and $j= \lfloor k(n/k)^d \rfloor$, there are at most $yk$ vertices in $[j]$ with fewer than $(1-x)m\log(n/j)$ neighbours in $[n]\setminus [j]$.
\end{lem}
\begin{proof}
Let us start with showing the desired property for $G^n_{m-\old}$. Pick a constant $A<\phi(-x)$ but sufficiently close to $\phi(-x)$ so that the following properties are satisfied:
\begin{equation}\label{eq:Aconstant}
(1-d)Am-d>0
\qquad\text{and}\qquad
\alpha < \left( \frac {0.99y}{e} \right)^{ 1 / ((1-d)Am-d) }.
\end{equation}
Let us concentrate on any $k$ in the range of consideration; that is, $1\le k\le \alpha n$. Let $X_v$ be the number of neighbours in $[n] \setminus [j]$ of a given vertex $v\in [j]$. By~\eqref{eq:m-old} (and the claim below), $X_v$ is a sum of independent Bernoulli random variables with parameter 
$$
1 - \left( 1 - \frac {1}{t-1} \right)^m \sim \frac {m}{t}
$$ 
for $t\in[n]\setminus[j]$ (note that $j \gg 1$). Hence, we get that 
\[
\ex [X_v] \sim \sum_{t=j+1}^{n}\frac{m}{t} \sim m\log(n/j).
\]
By the generalization of the Chernoff bound~(\ref{eq:strongChernoff2}) we get that the probability of $v$ having less than $(1-x)m\log(n/j)$ neighbours in $[n] \setminus [j]$ is at most 
\[
\exp\Big( - (1+o(1)) m\log(n/j)\varphi(-x) \Big) < (j/n)^{Am},
\]
since $A<\phi(-x)$. We will call such a vertex \emph{bad}. Using the fact that the events that two or more vertices are bad are negatively correlated, the probability of having at least $yk$ bad vertices in $[j]$ is at most
\[
\binom{j}{\lceil yk\rceil} (j/n)^{Am \lceil yk\rceil} 
\leq \left(\frac{ej(j/n)^{Am}}{\lceil yk\rceil}\right)^{\lceil yk\rceil}
\leq \left(\frac{e}{y} \left(\frac{k}{n}\right)^{(1-d)Am - d} \right)^{yk}
=: a_k.
\]
Note that, if $k+1 \le \eps n$ for a small constant $\eps > 0$,
\begin{align}
\frac{a_{k+1}}{a_k} 
&= \left(\frac{k+1}{k}\right)^{yk((1-d)Am-d)}
     \left(\frac{e}{y}\left(\frac{k+1}{n}\right)^{(1-d)Am-d}\right)^y
\notag\\
&\leq \exp\left(y((1-d)Am-d)\right) \left(\frac{e}{y}\right)^y \eps^{y((1-d)Am-d)}
\notag\\
&= \left(\frac{e}{y}\right)^y \exp\of{y((1-d)Am-d)(1+\log\eps)}<0.99,
\label{eq:aratio}
\end{align}
by the first inequality in~\eqref{eq:Aconstant} and provided that $\eps$ is sufficiently small.
Hence, $\sum_{k=1}^{\lfloor\eps n\rfloor} a_k = O(a_1) = o(1)$. On the other hand, 
\[
\sum_{k = \lfloor\eps n\rfloor+1}^{\lfloor\alpha n\rfloor} a_k 
\le \sum_{k = \lfloor\eps n\rfloor+1}^{\lfloor\alpha n\rfloor} \left(\frac{e}{y} \alpha^{(1-d)Am-d}\right)^{yk} \leq \sum_{k = \lfloor\eps n\rfloor+1}^{\lfloor\alpha n\rfloor} 0.99^{yk} = O(0.99^{\eps yn}) = o(1),
\]
by the second inequality in~\eqref{eq:Aconstant}.
It follows that the probability that the desired property does not hold for some $k$ is $o(1)$. The desired property holds for $G^n_{m-\old}$. 

\medskip

Adjusting the proof for $\pi_1 (G^n_{2m,m'})$ is straightforward, by using Lemma~\ref{lem:bound_no_edges} (with $\sigma=1$, $m_1=2m$ and $m_2=m'$) instead of~\eqref{eq:m-old} that we used above.
The only difference is that now $X_v$ is stochastically lower bounded by a sum of independent Bernoulli random variables with parameter 
$$
1 - \left( 1 - \frac {1}{2t} \right)^{2m} \sim \frac {m}{t}
$$ 
for $t\in[n]\setminus[j]$. It follows that $\ex [X_v] \ge (1+o(1)) m \log (n/j)$, and the rest of the proof continues as before.
\end{proof}

\bigskip

Next, we will use Lemma~\ref{lem:goodold} to derive an expansion property of $G^n_{m-\old}$ and $\pi_1 (G^n_{2m,m'})$ that will play a key role in the argument. We will show that for all sets of vertices $K$ of moderate size, $N(K)$ is large. (Recall the definition of $N(K)$ from Section~\ref{ssec:notation}.)

\begin{lem}\label{lem:expansion}
Let $\ell\in\{1,2\}$. Fix any constants $0<x,y,z,d,\alpha<1$ and $m\in\nat$ satisfying
\begin{gather}
y < z < 1 - \frac{\ell+1}{dm},
\label{eq:zrange}
\\
(1-d)\varphi(-x)m-d>0,
\label{eq:dphim}
\end{gather}
and
\begin{subnumcases}{\label{eq:alphabound}\alpha <}
\left(\frac{0.99y}{e}\right)^{\frac{1}{(1-d)\phi(-x)m-d }}
\label{eq:alphabound1}
\\
\exp\left( - \frac {\ell+1-z}{(1-x)(1-d)(z-y)} \right)
\label{eq:alphabound2}
\\
\frac{1}{\ell+1}  - \frac{1}{d(1-z)m}
\label{eq:alphabound3}
\\
\exp\left( - \dfrac{ (1-z)m  \log(\ell+1) + \ell+1 - \ell\log\ell }{d (1-z)m  - \ell-1} \right).
\label{eq:alphabound4}
\end{subnumcases}
Then, the following holds a.a.s.\ for $G^n_{m-\old}$. 
Every set of vertices $K$ with $1\le |K|\le \alpha n$ satisfies $|N(K)|\ge \ell |K|$.

Moreover, suppose that we can replace~\eqref{eq:alphabound3} and~\eqref{eq:alphabound4} by the following stronger conditions
%
% I want the next two subequations to be labelled the same as
% \eqref{eq:alphabound3} and \eqref{eq:alphabound4} with an additional prime
%
\newcounter{my_counter}
\renewcommand\theequation{\ref{eq:alphabound}$'$\alph{my_counter}}
\begin{numcases}{\alpha <}
\setcounter{my_counter}{3}
\frac{1}{\ell+1}  - \frac{8/3}{d(1-z)m}
\label{eq:alphabound5}
\\
\stepcounter{my_counter}
\exp \left( - \frac{(1-z)m \log \left((2 + w)^2(\ell+1) \right)  + \ell + 1 - \ell\log\ell)}{d(1-z)m  - (\ell+1)} \right).
\label{eq:alphabound6}
\end{numcases}
%
% Here I reset the equation labelling to normal
%
\addtocounter{equation}{-2}
\renewcommand\theequation{\arabic{equation}}
where 
\begin{equation}\label{eq:wbound}
w = \sqrt{\frac{8}{m}}\log\left(e\left(1+\frac{1}{(\ell+1)\alpha^d}\right)\right).
\end{equation}
Then, the previous a.a.s.\ statement is also a.a.s.\ valid for $\pi_1 (G^n_{2m,m'})$ (for any $m'\in\ent_{\ge0}$).
\end{lem}
Note that conditions~(\ref{eq:zrange}--\ref{eq:wbound}) can be trivially satisfied by taking any arbitrary $x,y,z,d\in(0,1)$ and $w>0$ such that $y<z$, and then choosing $m\in\nat$ sufficiently large and $\alpha>0$ sufficiently small. We leave the result with such flexibility as our goal will be to tune everything to get the constant $m$ as small as possible. 

Moreover, let us note that we assumed that $\ell \in \{1,2\}$ as this is what will be used in our specific application of this lemma. However, it is straightforward to verify that the lemma holds for any real number $\ell > 0$ (by placing floors and ceilings in the appropriate places of the argument).
\begin{proof}
Let us start with showing the desired property for $G^n_{m-\old}$. A set $K$ of vertices of size $k$ is of type~1 if it contains at least $zk$ vertices in $[j]$, where $j=\lfloor k (n/k)^d \rfloor$. Otherwise, $K$ is of type~2.
First, we will focus on sets of type~1.
Observe that constants $x,y,d,\alpha,m$ satisfy the requirements of Lemma~\ref{lem:goodold} (by assumptions~\eqref{eq:dphim} and~\eqref{eq:alphabound1}). Since our aim is to obtain a statement that holds a.a.s., we may assume that the conclusion of Lemma~\ref{lem:goodold} holds, and proceed to prove the desired statement (deterministically) for any given set of type~1.
Let $K$ be a fixed set of type~1 of size $1\le k\le \alpha n$.
It follows that at least $(z-y)k$ vertices in $K\cap [j]$ have at least 
\[
(1-x) m \log(n/j)  \ge
(1-x)(1-d) m \log(1/\alpha)  > \frac{\ell+1-z}{z-y}m
\]
neighbours in $[n] \setminus [j]$ (by assumption~\eqref{eq:alphabound2}). Denote this set of at least $(z-y)k$ vertices by $K_0$, and let $K'_0$ be the set of neighbours of $K_0$ in $[n] \setminus [j]$. Looking at the degrees of the bipartite graph induced by the parts $K_0$ and $K'_0$, we conclude that $|K'_0| \ge \frac {\ell+1-z}{z-y} |K_0| \ge (\ell+1-z) k$.
Therefore, $|N(K)| \ge |K'_0\setminus K| \ge |K'_0| - |K \setminus [j] | \ge (\ell+1-z)k - (1-z)k = \ell k$, as desired.
This proves the claim for all sets of type~1.

Now, consider a set $K$ of size $1\le k\le \alpha n$ of type~2 (hence, $K$ must contain at least $(1-z)k$ vertices in $[n] \setminus [j]$). Observe that if $|N(K)| < \ell k$, then there must exist a set $S=S(K)$ of size $\ell k-1$ with $S\cap K=\emptyset$ such that all vertices in $K \setminus [j]$ have all of their neighbours in $K\cup S$. An important observation is that, when a vertex in $K \setminus [j]$ generated $m$ edges to older vertices, all vertices in $[j]$ were available for possible destinations.  Hence, the probability that one such vertex has all the neighbours in $K \cup S$ is at most
\begin{equation}\label{eq:10}
\left(\frac{(\ell+1)k-1}{j}\right)^m
\le \left(\frac{(\ell+1)k-1}{k (n/k)^d - 1}\right)^m
\le \left((\ell+1) (k/n)^d\right)^m.
\end{equation}
As a result, the expected number of sets $K$ of type~2 and size $k$ with $|N(K)| < \ell k$ is at most
\begin{align}
b_k &:= \binom{n}{k} \binom{n-k}{\ell k} \left((\ell+1)(k/n)^d\right)^{(1-z)km}
\label{eq:bk}
\\
&\le \left(\frac{en}{\ell^{\ell/(\ell+1)} k}\right)^{(\ell+1) k}  \left((\ell+1)(k/n)^d\right)^{(1-z)km}
\notag\\
&= \left( \left(\frac{e^{\ell+1}}{\ell^{\ell} }\right)   (\ell+1)^{(1-z)m} \;  (k/n)^{(1-z)dm - \ell - 1} \right)^k =: b'_k.
\notag
\end{align}
Note that~\eqref{eq:zrange} implies that
the exponent of $k/n$ in the expression above is positive.
Hence, proceeding analogously as in~\eqref{eq:aratio}, we can show that there exists a small enough constant $\eps > 0$ such that $b'_{k+1}/b'_k < 0.99$ whenever $k+1 \le \eps n$. This implies that
\begin{equation}\label{eq:bprimebound}
b'_k \le b'_1(0.99)^{k-1}
\qquad\text{for $1\le k\le\eps n$},
\end{equation}
and so $\sum_{k=1}^{\lfloor\eps n\rfloor} b'_k = O(b'_1) = o(1)$. It follows that a.a.s.\ no set of type~2 and size at most $\eps n$ fails to have the desired property. To deal with larger sets,
it is more convenient to estimate $b_k$ directly instead of $b'_k$. For $\eps n \le k \le \alpha n$,
we use Stirling's formula ($s! \sim \sqrt{2\pi s} (s/e)^s$, as $s\to\infty$) and obtain
\[
b_k =  \frac{ \Theta(1/n) \left((\ell+1) (k/n)^d\right)^{(1-z)m k}}{(k/n)^{k} (\ell k/n)^{\ell k} (1-(\ell+1)k/n)^{n-(\ell+1)k}}  
= e^{n(f( k/n ) +o(1))},
\]
where
\begin{equation}\label{eq:frho}
f(\rho) = (1-z)m\rho \log \left((\ell+1) \rho^d\right) - \rho \log \rho
- \ell \rho\log (\ell \rho) - ({1-(\ell+1)\rho})\log (1-(\ell+1)\rho).
\end{equation}
(Note that in the above estimate of $b_k$ we implicitly used the fact that $n>(\ell+1)k$, which follows from~\eqref{eq:alphabound3}.) The second derivative of $f$ is
\begin{equation}\label{eq:fpp}
f''(\rho) = \frac{md(1-z)(1-(\ell+1)\rho) -(\ell+1)}{\rho (1-(\ell+1)\rho)},
\end{equation}
which is strictly positive for $\rho\in (0,\alpha]$ (by~\eqref{eq:alphabound3}), so $f$ is concave up in that interval. Therefore, 
\[
b_k \le e^{o(n)} \max\left\{ b_{\lfloor\eps n\rfloor}, b_{\lfloor\alpha n\rfloor} \right\} = (1+o(1))^n \max\left\{ b_{\lfloor\eps n\rfloor}, b_{\lfloor\alpha n\rfloor} \right\}
\qquad
\text{for all $\eps n\le k\le \alpha n$.}
\]
From~\eqref{eq:bprimebound}, we get that $b_{\lfloor\eps n\rfloor} \le b'_1 (0.99)^{\eps n-2} = o( (0.99)^{\eps n} )$. Next, we proceed to bound $b_{\lfloor\alpha n\rfloor} = e^{n(f(\alpha)+o(1))}$.
Observe that $0<\alpha<1/(\ell+1)$ by~\eqref{eq:alphabound3} and, for every $\alpha$ in that range, elementary analysis shows that
\begin{equation}\label{eq:tay}
- ({1-(\ell+1)\alpha})\log (1-(\ell+1)\alpha) < (\ell+1)\alpha.
\end{equation}
Hence,
\[
f(\alpha) < (1-z)m\alpha \log \left((\ell+1) \alpha^d\right) - \alpha \log \alpha
- \ell \alpha\log (\ell \alpha) + (\ell+1)\alpha <0,
\]
where the first inequality holds by~\eqref{eq:tay} and the second one by~\eqref{eq:alphabound4}. Putting everything together, for each $\eps n\le k\le \alpha n$,
\[
b_k \le \big(0.99+o(1)\big)^{\eps n}  +  \left(e^{f(\alpha)}+o(1)\right)^n = o(1/n).
\]
It follows that $\sum_{k=\lfloor\eps n\rfloor}^{\lfloor\alpha n\rfloor} b_k = o(1)$, and so a.a.s.\ all sets of type 2 of size at most $\alpha n$ satisfy the desired property, and the proof is finished for $G^n_{m-\old}$.

\medskip

Adjusting the proof for $\pi_1 (G^n_{2m,m'})$ is more complicated than in the previous lemma. As Lemma~\ref{lem:goodold} can be applied to $\pi_1 (G^n_{2m,m'})$, the proof for type 1 sets is not affected.
For type 2 sets, we need to use Lemmas~\ref{lem:total_weight} and~\ref{lem:extra} in order to obtain an analogue of~\eqref{eq:bk}.
Let $H$ be the a.a.s.\ event in the statement of Lemma~\ref{lem:total_weight} (replacing $m$ by $2m+m'$) with $\omega := \log \lfloor n^d \rfloor \to \infty$. As usual, as we aim for a statement that holds a.a.s., we may assume that event $H$ holds. Consider any set $K$ of size $1\le k\le \alpha n$ of type~2 and any set $S$ of size $\ell k-1$ such that $S\cap K=\emptyset$. Put as before $j=\lfloor k (n/k)^d \rfloor$, so $j \ge \lfloor n^d \rfloor \to \infty$, as $n \to \infty$ (and also $\log j \ge \omega$).
Setting $R=K\setminus[j]$ and $Q=K\cup S$  in Lemma~\ref{lem:extra} (with $\sigma=1$, $m_1=2m$ and $m_2=m'$), we conclude that, conditional upon $H$, the probability $p$ that every vertex in $K\setminus[j]$ has all neighbours in $K\cup S$ with respect to graph $\pi_1 (G^n_{2m,m'})$ satisfies
\begin{align}
p &\le \left( \left(1+\frac{C}{|Q|}+o(1) \right) \left( 2 \sqrt{\frac{|Q|}{j}}  + \sqrt{\frac{8|Q|}{mj}}\log\left(e\left(1+\frac{j}{|Q|}\right)\right)  \right) \right)^{2m|R|}
\notag\\
&\le \left( \left(1+\frac{C}{k}+o(1) \right) 2 \sqrt{(\ell+1)(k/n)^d } \left(1 + \sqrt{\frac{2}{m}}\log\left(e\left(1+\frac{(n/k)^d}{\ell+1}\right)\right) \right)  \right)^{2m|R|},
\label{eq:10b}
\end{align}
where we used that $|Q|/j \le \frac{(\ell+1)k-1}{k(n/k)^d-1} \le (\ell+1) (k/n)^d$ and
the fact that $\frac{\log(e(1+t))}{\sqrt t}$ is decreasing in $t\in (0,+\infty)$ (as observed below~\eqref{eq:PEH2}).
Since $|R|\ge(1-z)k$, replacing $|R|$ by $(1-z)k$ in~\eqref{eq:10b} gives a valid bound (even in the case that the base of the power in the right-hand side of~\eqref{eq:10b} is greater than $1$, since $p$ is a probability). Hence, we conclude that
\begin{align}
p &\le \left( \left(1+\frac{C}{k}+o(1) \right)^2 4 (\ell+1)(k/n)^d  \left(1 + \sqrt{\frac{2}{m}}\log\left(e\left(1+\frac{(n/k)^d}{\ell+1}\right)\right) \right)^2  \right)^{(1-z)km}
\label{eq:10c}
\\
&\le \left( C'  (k/n)^d  \log^2(n/k) \right)^{(1-z)km},
\label{eq:10d}
\end{align}
for some constant $C'>0$ that may depend on $C$, $\ell$ and $\alpha$. The role of $p$ (and its bounds~\eqref{eq:10c} and~\eqref{eq:10d}) will be very similar to that of $\left((\ell+1)(k/n)^d\right)^{(1-z)km}$ in the computations that we did for the $G^n_{m-\old}$ model (see~\eqref{eq:10} and~\eqref{eq:bk}).
Indeed, for the $\pi_1 (G^n_{2m,m'})$ model, the expected number (conditional upon $H$) of sets $K$ of type~2 and size $k$ with $|N(K)| < \ell k$ is at most
\begin{align}
b_k &:= \binom{n}{k} \binom{n-k}{\ell k} p
\label{eq:bk2}
\\
&\le \left( \left(\frac{e^{\ell+1}}{\ell^{\ell} }\right)  C'   (k/n)^{(1-z)dm - \ell - 1}  \log^{2m(1-z)}(n/k) \right)^k =: b'_k.
\notag
\end{align}
Our goal is to show that $\sum_{k=1}^{\lfloor\alpha n\rfloor} b_k=o(1)$. This, combined with the fact that $H$ is a.a.s.\ true, will yield the desired result for $\pi_1 (G^n_{2m,m'})$, and finish the proof of the lemma. The argument to bound $b_k$ (or $b'_k$) is analogous to the one we used for $G^n_{m-\old}$, so we will only sketch the main differences.
By inspecting the ratios $b'_{k+1}/b'_k$, we can easily check that~\eqref{eq:bprimebound} remains valid in the present context for $1\le k\le\eps n$ provided that $\eps>0$ is a sufficiently small constant given $C'$, $m$, $\ell$, $z$ and $d$. Therefore $\sum_{k=1}^{\lfloor\eps n\rfloor} b'_k = O(b'_1) = o(1)$ as before. To analyze the case $\eps n\le k\le \alpha n$, we plug the bound obtained in~\eqref{eq:10c} into the definition of $b_k$ in~\eqref{eq:bk2}. Applying Stirling's formula to the resulting bound and performing elementary manipulations, we conclude that
\[
b_k \le e^{n( \hat f(k/n) +o(1) )},
\]
where $\hat f(\rho)=f(\rho)+g(\rho)$ with $f(\rho)$ defined in~\eqref{eq:frho} and
\[
g(\rho) :=
2(1-z)m\rho
\log \left(2 + \sqrt{\frac{8}{m}}\log\left(e\left(1+\frac{1}{(\ell+1)\rho^d}\right)\right) \right).
\]
The function $\hat f$ will play the same role as $f$ did in the argument for the for $G^n_{m-\old}$ model.
In order to adapt the previous proof to the current model and show that $\sum_{k=\lfloor\eps n\rfloor}^{\lfloor\alpha n\rfloor} b_k = o(1)$, we only need to check that $\hat f(\alpha)<0$ and $(\hat f)''(\rho)>0$ for all $\rho\in(0,\alpha]$. We proceed to verify these two claims. Firstly, from~\eqref{eq:wbound}, we get that
\[
g(\alpha) = (1-z)m\alpha
\log(2 + w)^2.
\]
Combining this and~\eqref{eq:tay} yields
\[
\hat f(\alpha) < (1-z)m\alpha \log \left((2 + w)^2(\ell+1) \alpha^d\right)
- \alpha \log \alpha - \ell \alpha\log (\ell \alpha) + (\ell+1)\alpha
<0,
\]
where the second inequality follows directly from~\eqref{eq:alphabound6}.
Secondly, we differentiate $g$ twice and obtain
\[
g''(\rho) =
- \frac{2md(1-z) \sqrt{8/m}}{\rho \left(1 + (\ell+1)\rho^d \right) (2+\lambda)} \left(  1
+ \frac{d\sqrt{8/m}}{\left(1 + (\ell+1)\rho^d \right)(2+\lambda)}
- \frac{ d(\ell+1)\rho^d}{1 + (\ell+1)\rho^d} \right),
\]
where
\[
\lambda = \lambda(\rho) = \sqrt{\frac{8}{m}}\log\left(e\left(1+\frac{1}{(\ell+1)\rho^d}\right)\right).
\]
Note that, for all $\rho\in(0,\alpha]$,
\[
\frac{\sqrt{8/m}}{\left(1 + (\ell+1)\rho^d \right)(2+\lambda)}
= \frac{1}{\left(1 + (\ell+1)\rho^d \right) \left(\sqrt{m/2}+\log\left(e\left(1+\frac{1}{(\ell+1)\rho^d}\right)\right) \right)} \le 1/4,
\]
since $m\ge1$ and
$\left(1 + t \right) \left(\sqrt{1/2}+\log\left(e\left(1+1/t\right)\right) \right) \ge 4.12 >4$ for all $t\in(0,\infty)$ (by elementary analysis). Therefore,
\begin{align*}
g''(\rho) &\ge
- \frac{2md(1-z) \sqrt{8/m}}{\rho \left(1 + (\ell+1)\rho^d \right) (2+\lambda)} \left(  1
+ \frac{\sqrt{8/m}}{\left(1 + (\ell+1)\rho^d \right)(2+\lambda)} \right)
\\
&\ge
- \frac{2md(1-z)}{\rho}\, (1/4)(1+1/4)
\\
&\ge
- (5/8)\, \frac{md(1-z)}{\rho}.
\end{align*}
Combining this bound with~\eqref{eq:fpp}, we conclude that
\[
(\hat f)''(\rho) = f''(\rho) + g''(\rho) \ge 
\frac{(3/8)\, md(1-z)(1-(\ell+1)\rho) -(\ell+1)}{\rho (1-(\ell+1)\rho)},
\]
which is positive for all $\rho\in(0,\alpha]$ by~\eqref{eq:alphabound5}. This finishes the proof of the lemma for the $\pi_1 (G^n_{2m,m'})$ model.
\end{proof}

\bigskip

Finally, we include some properties of large sets of vertices, whose size is not covered in the previous expansion lemma. We will show that (a.a.s.\ in $G^n_{m-\old}$ or $\pi_1 (G^n_{2m,m'})$) large sets of vertices must induce some edges and disjoint pairs of large sets must have some edges across. 
These properties will guarantee the existence of a long path and a large matching that will be extended later to a Hamilton cycle and a perfect matching, respectively.

\begin{lem}\label{lem:bigpair}
For any constant $m \ge 12$, let $\beta = \beta(m) \in (0,1/2)$ be such that
$$
2 \beta \log \beta + (1-2\beta) \log (1-2\beta) + \beta^2 m / 4 = 0,
$$ 
and for any constant $m \ge 1$, let $\gamma = \gamma(m) \in (0,1)$ be such that
$$
\gamma \log \gamma + (1-\gamma) \log (1-\gamma) + \gamma^2 m / 2 = 0.
$$
Fix $m,m'\in\ent_{\ge0}$.
Then, the following two properties hold a.a.s.\ for $G^n_{m-\old}$ and for $\pi_1 (G^n_{2m,m'})$.
\begin{itemize}
\item [(i)]
If $m \ge 12$, then there is no pair of disjoint sets of vertices $A,B$, both of size at least $\beta n$, with no edges between $A$ and $B$. As a result, the length of a longest path is at least $n- 2\lceil\beta n\rceil \sim (1-2\beta)n$.
\item [(ii)] If $m \ge 1$, then  no set $C$ of size at least $\gamma n$ forms an independent set. As a result, the size of a maximum matching is at least $(1-\gamma)n/2$.
\end{itemize} 
\end{lem}

Before we prove the lemma, let us make a few observations. First, the lower bound of $12$ on $m$ is needed for $\beta=\beta(m)$ to be well defined. Second, note that the function $f(x) = - 2 x \log x - (1-2x) \log (1-2x)$ is maximized at $f(1/3)=\log 3$. Hence, $\beta \le \sqrt{ 4 \log 3 / m}$ and, in particular, $\beta$ tends to zero as $m \to \infty$. In fact, one can show that $\beta \sim 8 \log m / m$ as $m\to\infty$. Similarly, we get that $\gamma \le \sqrt{2 \log 2 / m}$, and $\gamma \sim 2 \log m / m$ as $m\to\infty$. Finally, let us mention that the second part of part (i) uses simple ideas that proved to be extremely useful in many current applications (see, for example,~\cite{DudekPralat, DudekPralat2, DudekPralat3, L15, P14}). 
Such techniques were used for the first time in~\cite{Krivelevich1, Krivelevich2} (see the recent book~\cite{Krivelevich_book} that covers several tools including this one, or another recent book on random graphs~\cite[Chapter~6.3]{BookFK}).

\begin{proof}
Let us start with the $G^n_{m-\old}$ model and part (i). Consider any pair of disjoint sets $A$ and $B$, both of size $\lceil \beta n\rceil$. Let $U$ be the set of $(|A|+|B|)/2 = \lceil \beta n\rceil$ oldest vertices in $A\cup B$ and let $U'= (A\cup B)\setminus U$ contain the youngest half. Without loss of generality, $A$ contains at least $\lceil\beta n/2\rceil$ vertices in $U$ and $B$ contains at least $\lceil\beta n/2\rceil$ vertices in $U'$. Therefore, the probability that there are no edges between $A\cap U$ and $B\cap U'$ is at most
\begin{equation}\label{eq:15}
(1-\lceil\beta n/2\rceil/n)^{m\lceil\beta n/2\rceil} \le (1-\beta/2)^{m\beta n/2} \le \exp(-\beta^2 m n/4)
\end{equation}
by~\eqref{eq:m-old}. Hence, using Stirling's formula ($s! \sim \sqrt{2 \pi s} (s/e)^s$, as $s \to \infty$), the expected number of pairs of sets $A,B$ that do not have the desired property is at most
\[
{n \choose \lceil\beta n\rceil} {(n-\lceil\beta n\rceil \choose \lceil\beta n\rceil} \exp(-\beta^2 m n/4)
\;=\; o \left( \beta^{-2\beta n} (1-2\beta)^{-(1-2\beta)n} \exp(-\beta^2 m n / 4) \right) \;=\; o(1),
\]
by the definition of $\beta$. The first claim of part (i) follows by Markov's inequality.

The second claim of part (i) follows from the first claim (deterministically). For a contradiction, suppose that the first claim holds and there is no path of length $h = n-2\lceil\beta n\rceil$ (or equivalently all paths contain at most $h$ vertices). We perform the following algorithm and construct a path $P$. Let $v_1$ be an arbitrary vertex. Initially, let $P=(v_1)$, $U = V \setminus \{v_1\}$, and $W = \emptyset$. Then, if there is an edge from $v_1$ to $U$ (say from $v_1$ to $v_2 \in U$), we extend the path to $P=(v_1,v_2)$ and remove $v_2$ from $U$. We continue extending the path $P$ this way for as long as possible. Since there is no path of length $h \le n-1$, we must reach the point of the process in which $P$ cannot be extended, that is, there is a path from $v_1$ to $v_k$ ($k \le h$) and there is no edge from $v_k$ to $U$. This time, $v_k$ is moved to $W$ and we try to continue extending the path from $v_{k-1}$, reaching another critical point in which another vertex will be moved to $W$, etc. If $P$ is reduced to a single vertex $v_1$ and no edge to $U$ is found, we move $v_1$ to $W$ and simply re-start the process from another vertex from $U$, again arbitrarily chosen. 
An obvious but important observation is that during this algorithm there is never an edge between $U$ and $W$. Moreover, in each step of the process, the size of $U$ decreases by 1 or the size of $W$ increases by 1. Finally, the number of vertices of the path $P$ is always at most $h=n-2\lceil\beta n\rceil$ by assumption. Hence, at some point of the process both $U$ and $W$ must have size at least $\lceil \beta n\rceil$. But this contradicts the first claim of part (i) and so this part is finished. 

\smallskip

Now, let us move to part (ii). Consider any set $C$ of size $k = \lceil\gamma n\rceil$. Again using~\eqref{eq:m-old}, the probability $C$ forms an independent set is at most
\begin{equation}\label{eq:16}
\prod_{i=1}^{k-1} \left( 1 - \frac {i}{n} \right)^m \le \exp \left( - \frac {m}{n} \sum_{i=1}^{k-1} i \right) = \exp \left( - \frac {m k (k-1)}{2n} \right) = O \left( \exp \left( - \gamma^2 mn/2 \right) \right).
\end{equation}
The expected number of sets that do not have the desired property is of order at most 
$$
{n \choose \lceil \gamma n\rceil} \exp \left( - \gamma^2 mn/2 \right) = o \left( \gamma^{-\gamma n} (1-\gamma)^{-(1-\gamma)n} \exp (- \gamma^2 mn/2) \right) = o(1),
$$
by the definition of $\gamma$, and the first claim of part (ii) is proved. 

The second claim of part (ii) is now a trivial (deterministic) implication. Indeed, if there is no matching of size at least $(1-\gamma)n/2$, then any maximum (or even maximal) matching leaves at least $\gamma n$ vertices that are not matched that form an independent set. This would contradict the first claim and so the proof of the theorem for $G^n_{m-\old}$ is finished.
Adjusting the proof for $\pi_1 (G^n_{2m,m'})$ is trivial, and only requires using Lemma~\ref{lem:bound_no_edges} (with $\sigma=1$, $m_1=2m$ and $m_2=m'$) instead of~\eqref{eq:m-old} in order to obtain the analogues of~\eqref{eq:15} and~\eqref{eq:16}. 
\end{proof}

\bigskip

Tuning the constants in the previous lemmas requires some patience but is straightforward. (Maple or some other software might be helpful.) For example, we get the following set of constants for the $G^n_{m-\old}$ model:
\begin{lem}\label{lem:constants1}\hspace{0cm}
\renewcommand{\labelenumi}{(\alph{enumi})}
\begin{itemize}
\item [(a)]
The following constants satisfy conditions~(\ref{eq:zrange}--\ref{eq:alphabound}) in Lemma~\ref{lem:expansion}: $m=120$, $\ell=1$, $\alpha = 0.0538$,  $x=0.22791$, $y=0.020063$, $z=0.851649$ and $d=0.387967$. Moreover, for $m=120$ we have $\gamma = \gamma(m) \le 0.06238$ in Lemma~\ref{lem:bigpair}.
\item [(b)]\label{constantsb}
The following constants also satisfy  conditions~(\ref{eq:zrange}--\ref{eq:alphabound}) in Lemma~\ref{lem:expansion}:
$m=2{,}900$, $\ell =2$, $\alpha = 0.032003$, $x=0.048929$, $y=0.003625$, $z=0.965269$ and $d=0.353628$.
Moreover, for $m=2{,}900$ we have $\beta = \beta(m) \le 0.014414$ in Lemma~\ref{lem:bigpair}.
\end{itemize}
\end{lem}
Similarly, for $\pi_1 (G^n_{2m,m'})$, we get:
\begin{lem}\label{lem:constants2}\hspace{0cm}
\begin{itemize}
\item [(c)]\label{constantsc}
The following constants satisfy conditions~(\ref{eq:zrange}--\ref{eq:wbound}) in Lemma~\ref{lem:expansion}:
$m=500$, $\ell=1$, $\alpha=0.016801$, $x=0.149159$, $y=0.008856$, $z=0.905885$ and $d=0.649188$.
Moreover, for $m=500$ we have $\gamma = \gamma(m) \le 0.019675$ in Lemma~\ref{lem:bigpair}.
\item [(d)]\label{constantsd}
The following constants also satisfy conditions~(\ref{eq:zrange}--\ref{eq:wbound}) in Lemma~\ref{lem:expansion}:
$m=14{,}000$, $\ell=2$, $\alpha = 0.008874$, $x=0.026228$, $y=0.001272$, $z=0.980855$ and $d=0.551906$.
Moreover, for $m=14{,}000$ we have $\beta = \beta(m) \le 0.003760$ in Lemma~\ref{lem:bigpair}.
\end{itemize}
\end{lem}

\subsection{Perfect matchings}

In this subsection, we deal with perfect matchings. Let us start with the following deterministic result that can be found, for example, in~\cite[Lemma~6.3]{BookFK}. Let $G=(V,E)$ be any graph. Let $A=A(G) \subseteq V$ be the set of vertices that are isolated by some maximum matching. For $v \in A$, let
$$
B(v) = \{w \in V \setminus \{v\}\colon \text{there exists a maximum matching that isolates both $v$ and $w$}\}.
$$
Observe that $B(v)\subseteq A$, and moreover $w \in B(v)$ implies that $v \in B(w)$. The set $B(v)$ is very important for understanding whether a few additional random edges, with $v$ as one of their endpoints, have a chance to increase the size of a maximum matching. Indeed, clearly, the size increases by one if an edge between $v$ and some vertex in $B(v)$ is added to the graph. 
Moreover, the following lemma holds. (We provide the proof for completeness.)

\begin{lem}[\cite{BookFK}]\label{lem:matching_expansion}
Let $G$ be a graph without a perfect matching. If $v \in A(G)$ and $B(v) \neq \emptyset$, then 
$$|N(B(v))| < |B(v)|.$$
\end{lem}

Note that the condition $B(v)\ne\emptyset$ in the lemma has the sole purpose of excluding the case that $G$ has odd order $n$ and a maximum matching of size $(n-1)/2$. Since our definition of perfect matching in this paper includes this situation, then the aforementioned condition is redundant.

\begin{proof}
Let $v$ be a vertex that is isolated by a maximum matching $M$.  Let $S_0=V(G)-(V(M)\cup\{v\})$, and observe that $S_0\ne\emptyset$ by our assumption on $B(v)$. Let $S_1$ be the set of vertices reachable from $S_0$ by a nontrivial even-length $M$-alternating path, and note that $S_0\cap S_1=\emptyset$.
Furthermore, if $w\in B(v)$, let $M'$ be a maximum matching isolating both $v$ and $w$.  If $w\notin S_0$, then in $M\triangle M'$ there is a path from $S_0$ to $w$, hence $w\in S_1$.  Therefore $B(v)=S_0\cup S_1$. 

Let $x\in N_G(S_1)$ and let $y\in S_1$ be a neighbour of $x$.
Let $P$ be a nontrivial even-length $M$-alternating path from $S_0$ to $y$.
Then $x\in V(M)$, as otherwise $P$ followed by $yx$ would be an $M$-augmenting path (contradicting the maximality of $M$).
So let $y'$ satisfy $xy'\in M$.  We will show that $y'\in S_1$; this implies that $M$ defines an injection $N_G(S_1)\to S_1$ by $x\mapsto y'$, and hence $|N_G(S_1)|\leq |S_1|$.
Note that if $x$ or $y'$ lie on $P$, then $xy'$ is an edge of $P$, so since $x\notin S_1$ we may truncate $P$ to end with $xy'$ and we are done.
Otherwise, we may extend $P$ by $yxy'$.

Finally, note that $N_G(S_0)\subseteq S_1\cup N_G(S_1)$ since each $z\in N_G(S_0)$ is adjacent in $M$ to some $w\in S_1$. 
%\xcb{I recall that Ben found a wrong statement somewhere around here. Is it the following one in brackets? If so, please delete it.} (in fact, the stronger statement $N_G(S_0)\subseteq N_G(S_1)$ holds, but we do not need it)
Thus
\begin{align*}
N_G(B(v)) = N_G(S_0\cup S_1) = (N_G(S_0)\cup N_G(S_1)) \setminus (S_0\cup S_1)
\subseteq N_G(S_1) \setminus S_0
\end{align*}
and so
\[ |B(v)| = |S_0|+|S_1| > |S_1| \ge |N_G(S_1)| > |N_G(B(v))|. \qedhere\]
\end{proof}

Now, we are ready to prove the main result for perfect matchings. Recall that we say that a graph on $n$ vertices has a perfect matching if it has a matching of size $\lfloor n/2 \rfloor$. 

\begin{thm}\label{thm:matchings-old}
Let $m \ge 159$. Then a.a.s.\ $G^n_{m-\old}$ has a perfect matching.
\end{thm}
\begin{proof}
We are going to use the two-round exposure technique, as discussed in Section~\ref{ssec:two-round}. Recall that the graph $G^n_{m-\old}$ can be seen as a union of two independent graphs: $G^n_{m_1-\old}$ and $G^n_{m_2-\old}$ as long as $m = m_1 + m_2$. In our situation, $m_1 = 120$ and $m_2 = m - m_1 \ge 39$. It follows from Lemma~\ref{lem:expansion}, Lemma~\ref{lem:bigpair}(ii) and Lemma~\ref{lem:constants1}(a) that the properties stated there hold a.a.s.\ for $G^n_{m_1-\old}$ with $\ell = 1$, $\alpha = 0.0538$ and $\gamma = \gamma(m_1) \le 0.06238$. In particular, as we aim for an a.a.s.\ statement, we may expose the edges of $G^n_{m_1-\old}$ first and assume that every set of vertices $K \subseteq [n]$ with $1\le |K|\le \alpha n$ satisfies $|N(K)|\ge |K|$, and that $G^n_{m_1-\old}$ has a matching that isolates at most $\gamma n$ vertices. It is an immediate but crucial observation that these two assumptions will still hold if we add extra edges to $G^n_{m_1-\old}$.

Now, we will repeatedly expose some edges of $G^n_{m_2-\old}$ (in some specific order) and add them to $G^n_{m_1-\old}$.
At each stage of the process, let $G$ denote the current graph, consisting of $G^n_{m_1-\old}$ together with the exposed edges from $G^n_{m_2-\old}$. We will argue that, if $G$ does not have a perfect matching, then there is a good chance that exposing more edges of $G^n_{m_2-\old}$ increases the size of a maximum matching by 1. Consider $A=A(G)$ and suppose that $G$ has no perfect matching. As all small sets expand well in $G \supseteq G^n_{m_1-\old}$, it follows from Lemma~\ref{lem:matching_expansion} that $|B(v)| \ge \alpha n$ for any $v \in A$ (note that $B(v) \neq \emptyset$, as no perfect matching is found yet). Therefore, recalling that $A\supseteq B(v)$, we conclude that $|A| \ge \alpha n$.

Suppose that we are at the first stage of the process (that is, when $G=G^n_{m_1-\old}$), and let us focus on the youngest vertex $v$ in $A$. We expose all edges of $G^n_{m_2-\old}$ from $v$ to older vertices and add them to $G$. Note that if we find some edge between $v$ and $B(v)$, then we increase the size of a maximum matching by one. As all vertices in $B(v)$ are older than $v$, the probability of increasing the size of a maximum matching is at least $1 - \left( 1 - \alpha \right)^{m_2}$ (see~\eqref{eq:m-old}). We update $G$ (including all exposed edges from $G^n_{m_2-\old}$) and also the set $A=A(G)$, and continue the process moving to another vertex $v$, the youngest vertex in $A$ that is not exposed yet with respect to $G^n_{m_2-\old}$. As before $|B(v)| \ge \alpha n$ but, perhaps, one vertex from $B(v)$ (the one that is already exposed) is younger than $v$. Hence, the probability of increasing the size of a maximum matching is at least $1 - \left( 1 - (\alpha n - 1)/n \right)^{m_2}$. In general, as long as we expose edges from $t$ vertices and a perfect matching is not found, the probability of extending a maximum matching after exposing another vertex is at least $1 - \left( 1 - (\alpha n - t)/n \right)^{m_2}$. 

It remains to estimate the probability that the process exposes edges of $\lfloor \alpha n\rfloor$ vertices and no perfect matching is found. If this happens, then the random variable $X$ that counts the number of times the process extends some maximum matching is smaller than $\gamma n / 2$. Note that $X$ is stochastically bounded from below by a random variable $Y$ that is a sum of independent Bernoulli random variables with parameter $1 - \left( 1 - (\alpha n - t)/n \right)^{m_2}$, for $t \in \{0,1,\ldots,\lfloor\alpha n\rfloor-1 \}$. Using the Euler--Maclaurin formula and the change of variable $x=t/n$, we can approximate $\ex Y$ by an integral and obtain
\begin{eqnarray}
\frac {\ex [Y]}{n} &\sim& \int_0^\alpha \Big( 1- (1- (\alpha - x))^{m_2} \Big)\,dx \nonumber \\
&=& \left[ x - \frac {(1-\alpha+x)^{m_2+1}}{m_2+1} \right]_{0}^{\alpha}  = \alpha - \frac {1}{m_2+1} + \frac {(1-\alpha)^{m_2+1}}{m_2+1} > \frac {\gamma}{2}. \label{eq:success1}
\end{eqnarray}
It follows from the generalized Chernoff bound that a.a.s.\ the process does not fail and a perfect matching is found.
\end{proof}

Adjusting the result to $G^n_m$ is straightforward, given all the tools that we developed in Section~\ref{ssec:relationship}.

\begin{cor}\label{cor:matchings}
Let $m \ge 1,260$. Then a.a.s.\ $G^n_{m}$ has a perfect matching.
\end{cor}
\begin{proof}
The argument is analogous to that of Theorem~\ref{thm:matchings-old}, so we will only sketch the differences.
Let $m_1 = 500$, and $m_2 = m - 2m_1 \ge 260$. It follows from Lemma~\ref{lem:expansion}, Lemma~\ref{lem:bigpair}(ii) and Lemma~\ref{lem:constants2}(c) that the properties stated there hold a.a.s.\ for $\pi_1\big(G^n_{2m_1,m_2}\big)$ with $\ell = 1$, $\alpha = 0.016801$ and $\gamma =\gamma(m_1) \le 0.019675$. 
In particular, $\pi_1\big(G^n_{2m_1,m_2}\big)$ has the desired properties as we claimed before for $G^n_{120-\old}$ in the proof of Theorem~\ref{thm:matchings-old}. The second part of the argument has to be adjusted slightly, using Lemma~\ref{lem:bound_no_edges_future} instead of~\eqref{eq:m-old}. This time, $X$ is stochastically bounded from below by a random variable $Y$ that is a sum of independent Bernoulli random variables with parameter $1 - \left( 1 - (\alpha n - t)/(2n) \right)^{m_2}$, $t \in \{0,1,\ldots,\lfloor\alpha n\rfloor-1 \}$. Now 
\begin{eqnarray}
\frac {\ex [Y]}{n} &\sim& \int_0^\alpha \Big( 1- (1- (\alpha - x)/2)^{m_2} \Big)\,dx = \left[ x - \frac {2 (1-\alpha/2+x/2)^{m_2+1}}{m_2+1} \right]_{0}^{\alpha} \nonumber \\
&=& \alpha - \frac {2}{m_2+1} + \frac {2(1-\alpha/2)^{m_2+1}}{m_2+1} > \frac {\gamma}{2}, \label{eq:new_int}
\end{eqnarray}
and the rest of the proof continues unaltered.
\end{proof}

\subsection{Hamiltonian cycles}

In this subsection, we deal with Hamiltonian cycles. The argument is very similar to the one we used for perfect matchings so we skip some details. As before, we start with a deterministic result that can be found, for example, in~\cite[Corollary~6.7]{BookFK}. This approach was an important breakthrough in finding Hamiltonian cycles in random graphs and came with the result of P\'osa~\cite{Posa}. Let $G=(V,E)$ be any graph. Suppose that $P=(a, \ldots, x,y,y', \ldots, b',b)$ is a path and $bx$ is an edge where $x\ne b$ is either $a$ or an interior vertex of $P$. Then, the path $P' = (a, \ldots, x,b,b', \ldots, y', y)$ is said to be obtained from $P$ by a \emph{rotation} with vertex $a$ fixed. Now, for a given path $P$ with $a$ as one of its endpoints, let $\END(P,a)$ be the set of vertices $v$ such that there exists a path from $a$ to $v$ that is obtained from $P$ by a sequence of rotations with vertex $a$ fixed. Then the following lemma holds.  (And again we include the proof for completeness.)

\begin{lem}[\cite{BookFK}]\label{lem:cycles_expansion}
Let $G=(V,E)$ be a graph, $P$ be any longest path of $G$, and $a$ one of its endpoints. Then, 
$$|N(\END(P,a))| < 2 |\END(P,a)|.$$
\end{lem}
\begin{proof}
We claim that if $v\in V(P)\setminus \END(P,a)$ and $v$ is adjacent to some vertex $w$ in $\END(P,a)$, then there is $w'\in \END(P,a)$ such that $vw'\in E(P)$.  Indeed, consider the path $P_w$ witnessing $w\in \END(P,a)$.  Consider $x$ with $vx\in E(P)$; if $vx\notin E(P_w)$, then one of the rotations yielding $P_w$ from $P$ deleted the edge $vx$ and hence $x\in \END(P,a)$, so $w'=x$ satisfies the claim.  Otherwise $N_P(v)=N_{P_w}(v)$, and performing a rotation with the edge $vw$ in $P_w$ shows that one of the vertices of $N_{P_w}(v)$ is in $\END(P,a)$.

Now since $P$ is a longest path, $N_G(\END(P,a))\subseteq V(P)$.  Our claim implies that furthermore $N_G(\END(P,a))\subseteq N_P(\END(P,a))$.  Each vertex of $\END(P,a)$ has at most two neighbours along $P$, but also the endpoint of $P$ (other than $a$) has only one; hence $|N_P(\END(P,a))|<2|\END(P,a)|$.
\end{proof}

Suppose that $G$ is connected but has no Hamiltonian cycle. Let $A=A(G) \subseteq V$ be the set of vertices that are endpoints of some longest path. (In particular, the length of a longest path could be $n=|V|$.) For $v \in A$, let
$$
B(v) = \{w \in V \setminus \{v\}: \text{ $w \in \END(P,v)$ for some longest path $P$ of $G$ having $v$ as an endpoint}\}.
$$
This time, the set $B(v)$ is important for understanding whether a few additional random edges, with $v$ as one of their endpoints, have a chance to increase the length of a longest path (if there is no Hamiltonian path) or create a Hamiltonian cycle (if there is a Hamiltonian path). Indeed, if there is a Hamiltonian path, then adding an edge between $v$ and some vertex in $B(v)$ creates a Hamiltonian cycle. Otherwise, some longest path of length $k < |V|$ creates a cycle of length $k$. Since $G$ is connected, there is at least one edge joining a vertex from the cycle with some vertex outside and so a longer path can be created.

\bigskip

Now, we are ready to prove the main result for Hamiltonian cycles. Since the proof is very similar to the one of Theorem~\ref{thm:matchings-old}, we simply sketch it.

\begin{thm}\label{thm:cycles-old}
Let $m \ge 3{,}214$. Then a.a.s.\ $G^n_{m-\old}$ has a Hamiltonian cycle. 
\end{thm}
\begin{proof}
Again, we are going to use the two-round exposure: $G^n_{m-\old} = G^n_{m_1-\old} \cup G^n_{m_2-\old}$, with $m_1 = 2{,}900$ and $m_2 = m - m_1 \ge 314$. It follows from Lemma~\ref{lem:expansion}, Lemma~\ref{lem:bigpair}(i) and Lemma~\ref{lem:constants1}(b) that the properties stated there hold a.a.s.\ for $G^n_{m_1-\old}$ with $\ell = 2$, $\alpha = 0.032003$ and $\beta = \beta(m_1) \le 0.014414$. In particular, we may expose $G^n_{m_1-\old}$ and assume that every set of vertices $K \subseteq [n]$ with $1\le |K|\le \alpha n$ satisfies $|N(K)|\ge 2 |K|$, and that $G^n_{m_1-\old}$ has a path of length at least $n-2\lceil\beta n\rceil$. Moreover, we assume that $G^n_{m_1-\old}$ is connected by~\cite{BR}.

As in the proof of Theorem~\ref{thm:matchings-old}, we will sequentially pick vertices and expose the edges of $G^n_{m_2-\old}$ that connect them to older vertices. During this process, we try to extend the length of a longest path and, once a Hamiltonian path is created, to close the desired cycle. Each time, if the current graph does not have a Hamiltonian cycle yet, we update the set $A$ and all $B(v)$'s. As all small sets expand well, it follows from Lemma~\ref{lem:cycles_expansion} that $|B(v)| \ge \alpha n$ for every $v \in A$, and so also $|A| \ge \alpha n$. At each step, we choose the youngest vertex $v$ in $A$ that has not been picked yet, and expose the $m_2$ edges of $G^n_{m_2-\old}$ that go from $v$ to older vertices. As before, as long as we have only exposed edges from $t$ vertices and a Hamiltonian cycle has not been found, the probability of improving our current situation in the next step is at least $1 - \left( 1 - (\alpha n - t)/n \right)^{m_2}$. From an analogous computation to the one leading to~\eqref{eq:success1}, the expected number of successful steps in the process  is at least $(1+o(1))cn$, where
$$
c:= \alpha - \frac {1}{m_2+1} + \frac {(1-\alpha)^{m_2+1}}{m_2+1} > 2 \beta.
$$
Since it only takes $2\lceil\beta n\rceil \sim 2\beta n$ (or less) successful steps to discover a Hamiltonian cycle, we use the generalized Chernoff bound as before to conclude that a.a.s.\ a Hamiltonian cycle is found during the process.
\end{proof}

Finally, we will extend the result to $G^n_m$ for the last time, by adjusting the proof of Theorem~\ref{thm:cycles-old} in the same spirit as Corollary~\ref{cor:matchings} was obtained from Theorem~\ref{thm:matchings-old}. We omit details, and just state the main differences from the previous arguments.

\begin{cor}\label{cor:cycles}
Let $m \ge 29{,}500$. Then a.a.s.\ $G^n_{m}$ has a Hamiltonian cycle. 
\end{cor}
\begin{proof}
Let $m_1 = 14{,}000$, and $m_2 = m - 2m_1 \ge 1{,}500$. Take $\ell=2$ and $\alpha=0.008874$, and note that $\beta = \beta(m_1) \le 0.003760$.
Then, by Lemma~\ref{lem:expansion}, Lemma~\ref{lem:bigpair} and Lemma~\ref{lem:constants2}(d), $\pi_1\big(G^n_{2 m_1,m_2}\big)$ has the same desired properties as we claimed for $G^n_{2{,}900-\old}$ in the proof of Theorem~\ref{thm:cycles-old}. After adjusting~(\ref{eq:new_int}) and verifying that 
$$
\alpha - \frac {2}{m_2+1} + \frac {2(1-\alpha/2)^{m_2+1}}{m_2+1} > 2 \beta,
$$
the proof is finished.
\end{proof}

\section{Lower bound}\label{s:lower}

In~\cite{Robot2}, a simplified model for crawling complex networks such as the web graph was proposed, which is a variation of the robot vacuum edge-cleaning process introduced earlier by Messinger and Nowakowski~\cite{vacuum_MN}. In particular, it was shown that a.a.s.\ the preferential attachment model $G^n_2$ has the following property: the optimal robot crawler needs substantially more than $n$ steps to visit all the vertices of the graph (see Theorem~10 in~\cite{Robot2} and the discussion preceding it). This immediately implies (in a strong sense) that $G_2^n$ is not Hamiltonian a.a.s. Moreover, these observations, combined with an additional argument presented below, show a stronger property, namely, that $G^n_2$ has no perfect matching a.a.s. 

Here we briefly recall the the main ingredients of the argument in~\cite{Robot2} without reference to the robot crawling process.  Fix some constant $c\in (0,1)$ and consider $G^n_2$. We say that a vertex $t$ is \emph{old} if $t\leq cn$ (and \emph{young} otherwise).  Moreover, a vertex is called \emph{lonely} if it is never chosen as the recipient of an edge stemming from another vertex added later on in the process. (In particular, a lonely vertex must have degree $2$.)
Let $A_n$ be the number of vertices in $G^n_2$ that are young, lonely, with two old neighbours; let $B_n$ be the number of vertices that are young, lonely, with exactly one old neighbour; let $C_n$ be the number of vertices that are old and lonely; and finally, let $D_n$ count the number of vertices that are old but not lonely.
Pick any young vertex $t>cn$. It follows from Lemma~\ref{lem:total_weight_specific} that the probability that exactly $i$ edges stemming from $t$ are attached to old vertices is asymptotic to
\[
\binom{2}{i}\sqrt{cn/t}^i(1-\sqrt{cn/t})^{2-i}.
\]
Moreover, the probability that $t$ is lonely is equal to 
\[
\prod_{i=t+1}^n \left( 1 - \frac {2}{4i} + O\left(i^{-2}\right) \right)^2 = \exp \left( - \sum_{i=t+1}^n i^{-1} + O(t^{-1}) \right) \sim t/n,
\]
and this is still valid conditional on any event that involves only edges stemming from vertices in $[t]$. In particular, the probability that $t$ is lonely and has exactly one old neighbour is asymptotic to $2\sqrt{cn/t}(1-\sqrt{cn/t})t/n$, and thus
\[
\ex B_n \sim n \int_{c}^{1} 2\sqrt{c/x}(1-\sqrt{c/x})x\,dx
     = n \left( \frac{4\sqrt{c}}{3}-2c+\frac{2c^2}{3} \right).
\]
A similar calculation gives that $\ex\big(B_n(B_n-1)\big) \sim (\ex B_n)^2$ (since, for each pair of different young vertices $t_1$ and $t_2$, the events \{$t_1$ contributes to $B_n$\} and \{$t_2$ contributes to $B_n$\} are asymptotically independent). This implies that $\var B_n = o\big((\ex B_n)^2\big)$, and hence by Chebyshev inequality we get that a.a.s.\ $B_n\sim\ex B_n$. Proceeding analogously for $A_n$, $C_n$ and $D_n$, we conclude that a.a.s.:
\begin{align*}
A_n &\sim n \int_{c}^{1} (\sqrt{c/x})^2 x\,dx
     = n(c-c^2), \\
B_n &\sim n \int_{c}^{1} 2\sqrt{c/x}(1-\sqrt{c/x})x\,dx
     = n \left( \frac{4\sqrt{c}}{3}-2c+\frac{2c^2}{3} \right), \\
C_n &\sim n \int_0^c x\,dx = n \frac{c^2}{2},\\
D_n &\sim n \left( c-\frac{c^2}{2} \right).
\end{align*}

Consider the graph $H$ (on $n-D_n$ vertices) obtained by deleting the old but not lonely vertices.  The $C_n$ old and lonely vertices as well as the $A_n$ young and lonely vertices with two old neighbours are isolated in $H$, and the $B_n$ young and lonely vertices with one old neighbour have degree 1 in $H$.  Since
\[ \frac {B_n}{2n} + \frac {A_n}{n} + \frac {C_n}{n} \sim \frac{1}{2}\left(\frac{4\sqrt{c}}{3}-2c+\frac{2c^2}{3}\right) + c-c^2 + \frac{c^2}{2} > c-\frac{c^2}{2} \sim \frac {D_n}{n} \]
for any choice of $c \in (0,1)$, the graph a.a.s.\ does not have a Hamiltonian cycle. In fact, this argument \emph{almost} shows that a.a.s.\ there is  no perfect matching. Indeed, the number of odd components in $H$ is at least the number of isolated vertices in $H$, which in turn is a.a.s.\ equal to $A_n + C_n \sim (c-c^2/2)n$, and so it coincides with an asymptotic value of $D_n$, the number of vertices removed from $G^n_2$ to get $H$. Hence, it \emph{almost} violates Tutte's condition for existence of a perfect matching. 

In order to show that the necessary (and sufficient) Tutte's condition does not hold for, say, $c=1/4$, we are going to show that a.a.s.\ there are $\Omega(n)$ components of size 3 in $H$ (and so there are more odd components than vertices removed). As the constant hidden in the $\Omega()$ notation does not affect the argument, we are not going to optimize it here. Let us partition the vertex set $V=[n]$ into 4 subsets $V_1 = [\lfloor n/4\rfloor]$, $V_2 = [\lfloor n/2\rfloor] \setminus [\lfloor n/4\rfloor ]$, $V_3 = [\lfloor 3n/4\rfloor ] \setminus [\lfloor n/2\rfloor ]$, and $V_4 = [n] \setminus [\lfloor 3n/4\rfloor ]$. (In particular, $V_1$ is the set of old vertices.) We say that a triple $(v_2, v_3, v_4)$ forms a \emph{cherry} if $v_2\in V_2$, $v_3 \in V_3$ are the two older neighbours of a lonely vertex $v_4 \in V_4$, and $v_4$ is the only younger neighbour of each of $v_2$ and $v_3$. Moreover, we say that a cherry $(v_2, v_3, v_4)$ is \emph{sweet} if $v_2$ and $v_3$ both have two neighbours in $V_1$ (these neighbours are older than them; the third, common and younger, neighbour is $v_4$, of course). Let $S_n$ denote the number of sweet cherries.

We proceed to bound from below the probability that a given tuple $(v_2,v_3,v_4)$ forms a sweet cherry. First note that the probability that $v_2$ has two neighbours in $V_1$ is at least $(1+o(1))(1/2)^2$. Conditional upon that, the probability that no vertex $i$ ($v_2<i<v_3$) is adjacent to $v_2$ is at least $\prod_{i=v_2+1}^{v_3-1}\left(1-\frac{2}{4(i-1)}\right)^2$. Conditional upon these two events, the probability that $v_3$ has two neighbours in $V_1$ is at least $(1+o(1))(1/3)^2$. Conditional on all the previous events, the probability that no vertex $i$ ($v_3<i<v_4$) is adjacent to $v_2$ or $v_3$ is at least $\prod_{i=v_3+1}^{v_4-1}\left(1-\frac{4}{4(i-1)}\right)^2$. Conditional on all the above, the probability that $v_4$ is adjacent to both $v_2$ and $v_3$ is
$(2+o(1)) \left( \frac {2}{4n} \right)^2$. Finally, conditional upon all the previous events, the probability that no vertex $i$ ($v_4<i\le n$) is adjacent to $v_2$ or $v_3$ is at least $\prod_{i=v_4+1}^{n}\left(1-\frac{6}{4(i-1)}\right)^2$. Putting everything together, the probability that $(v_2,v_3,v_4)$ is a sweet cherry is at least
\[
(1+o(1))\, 2\left(\frac{1}{6}\right)^2 \left( \frac {1}{2n} \right)^2 \prod_{i=v_2+1}^{n} \left(1-\frac{3}{2(i-1)}\right)^2
= \Omega(1)
\left(\frac{1}{n}\right)^2 \left(\frac{v_2}{n}\right)^3
= \Omega\left(1/n^2\right).
\]
Since the number of possible tuples $(v_2,v_3,v_4)$ is $\Theta(n^3)$, we conclude that $\ex S_n = \Omega(n)$.
Moreover, given two different tuples $(v_2,v_3,v_4)$ and  $(v'_2,v'_3,v'_4)$,
easy calculations show that the events \{$(v_2,v_3,v_4)$ is a sweet cherry\} and \{$(v'_2,v'_3,v'_4)$ is a sweet cherry\} are asymptotically independent for $v_2\ne v'_2$, $v_3\ne v'_3$ and $v_4\ne v'_4$ or disjoint otherwise. Then, we can argue as we did before for $B_n$ to conclude that $\var S_n = o\big((\ex S_n)^2\big)$ and a.a.s.\ $S_n\sim\ex S_n$.
As a result, there are a.a.s.\ $\Omega(n)$ sweet cherries (inducing components of size 3 in $H$) and so $G^n_2$ a.a.s.\ has no perfect matching. 

Unfortunately, we do not know how to show that a.a.s.\ $G^n_{2-\old}$ has no perfect matching (perhaps it does have one, as some experiments for small values of $n$ might suggest). 
%\bc{In fact, some experiments suggest that maybe $G^n_{2-\old}$ does have a perfect matching.}\pc{Do you have a reference? If yes, we can add it. If not, then perhaps we can just add `Of course, the reason might be that it is simply not true, despite the fact that our upper bound is far away from 2.'}\bc{These were my SageMath experiments.  For $n=100$ and $n=1000$, I tested 10000 samples from 2-old; 8332 of the former had perfect matchings, and 9736 of the latter.}  
On the other hand, showing that a.a.s.\ it has no Hamiltonian cycle is easy: a.a.s.\ there are three lonely vertices that have a common neighbour.


\begin{thebibliography}{99}

\bibitem{AKS1} M.\ Ajtai, J.\ Koml\'os and E.\ Szemer\'edi, The first occurrence of Hamilton cycles in random graphs, \emph{Annals of Discrete Mathematics} \textbf{27} (1985) 173--178.

\bibitem{BBKMW11}
J.~Balogh, B.~Bollob{\'a}s, M.~Krivelevich, T.~M{\"u}ller, and M.~Walters.
\newblock Hamilton cycles in random geometric graphs.
\newblock {\em Ann. Appl. Probab.}, 21(3):1053--1072, 2011.

\bibitem{BA} A.L.\ Barab\'asi, R.\ Albert, Emergence of scaling in random networks, \emph{Science} \textbf{286} (1999) 509--512.

\bibitem{Krivelevich1} I.~Ben-Eliezer, M.~Krivelevich and B.~Sudakov, The size Ramsey number of a directed path, \emph{J. Combin. Theory Ser. B} \textbf{102} (2012), 743--755.

\bibitem{Krivelevich2} I.~Ben-Eliezer, M.~Krivelevich and B.~Sudakov, Long cycles in subgraphs of (pseudo)random directed graphs, \emph{J. Graph Theory} \textbf{70} (2012), 284--296.

\bibitem{BF} T.\ Bohman and A.M.\ Frieze, Hamilton cycles in 3-out, \emph{Random Structures and Algorithms} \textbf{35}, John Wiley and Sons, 393--417. 

\bibitem{B3} B.\ Bollob\'as, The evolution of sparse graphs, in Graph Theory and Combinatorics (Proceedings of the Cambridge Combinatorics Conference in Honour of Paul Erd\H{o}s (B.\ Bollob\'as; Ed)), Academic Press (1984) 35--57.

\bibitem{BF4} B.\ Bollob\'as and A.M.\ Frieze, On matchings and hamiltonian cycles in random graphs, \emph{Annals of Discrete Mathematics} \textbf{28} (1985) 23--46.

\bibitem{BR} B.\ Bollob\'as, O.\ Riordan, The diameter of a scale-free random graph, \emph{Combinatorica} \textbf{24} (2004) 5--34.

\bibitem{BRST} B.\ Bollob\'as, O.\ Riordan, J.\ Spencer, G.\ Tusn\'ady, The degree sequence of a scale-free random graph process, \emph{Random Structures and Algorithms} \textbf{18} (2001) 279--290.

%\bibitem{Robot1} A.\ Bonato, R.M.\ del Rio-Chanona, C.\ MacRury, J.\ Nicolaidis, X.\ P\'erez-Gimenez, P.\ Pra\l{}at, and K.\ Ternovsky, The robot crawler graph process, \emph{Discrete Applied Mathematics}, Preprint 2015. 

\bibitem{Robot2} A.\ Bonato, R.M.\ del Rio-Chanona, C.\ MacRury, J.\ Nicolaidis, X.\ P\'erez-Gimenez, P.\ Pra\l{}at, and K.\ Ternovsky, The robot crawler number of a graph, \emph{Proceedings of the 12th Workshop on Algorithms and Models for the Web Graph (WAW 2015)}, Lecture Notes in Computer Science \textbf{9479}, Springer, 2015, 132--147.

\bibitem{CF8} C.\ Cooper and A.M.\ Frieze, Hamilton cycles in random graphs and directed graphs, \emph{Random Structures and Algorithms} \textbf{16} (2000) 369--401.

\bibitem{CFR9} C.\ Cooper, A.M.\ Frieze and B.\ Reed, Random regular graphs of non-constant degree: connectivity and Hamilton cycles, \emph{Combinatorics, Probability and Computing} \textbf{11} (2002) 249--262.

\bibitem{DMP07}
J.~D{\'{\i}}az, D.~Mitsche, and X.~P{\'e}rez.
\newblock Sharp threshold for {H}amiltonicity of random geometric graphs.
\newblock {\em SIAM J. Discrete Math.}, 21(1):57--65 (electronic), 2007.

\bibitem{DudekPralat3} A.\ Dudek, F.\ Khoeini, and P.\ Pra\l{}at, Size-Ramsey numbers of cycles versus a path, manuscript.

\bibitem{DudekPralat} A.\ Dudek and P.\ Pra\l{}at, An alternative proof of the linearity of the size-Ramsey number of paths, \emph{Combinatorics, Probability and Computing} \textbf{24(3)} (2015), 551--555.

\bibitem{DudekPralat2} A.\ Dudek and P.\ Pra\l{}at,  On some multicolour Ramsey properties of random graphs, manuscript.

\bibitem{Book279} P.\ Erd\H{o}s and A.\ R\'enyi, On the existence of a factor of degree one of a connected random graph, \emph{Acta Mathematica Academiae Scientiarum Hungaricae} \textbf{17} (1966), 359--368.

\bibitem{FF11} T.I.\ Fenner and A.M.\ Frieze, On the existence of hamiltonian cycles in a class of random graphs, \emph{Discrete Mathematics} \textbf{45} (1983) 301--305.

\bibitem{F12} A.M.\ Frieze, Finding Hamilton cycles in sparse random graphs, \emph{Journal of Combinatorial Theory B} \textbf{44} (1988) 230--250.

\bibitem{BookFK} A.M.\ Frieze and M.\ Karo\'nski, \emph{Introduction to Random Graphs}, Cambridge University Press, 2015.

\bibitem{FL13} A.M.\ Frieze and T.\ \L{}uczak, Hamiltonian cycles in a class of random graphs: one step further, in Proceedings of Random Graphs '87, Edited by M.\ Karo\'nski, J.\ Jaworski and A.\ Ruci\'nski, John Wiley and Sons, 53--59.

\bibitem{JLR} S.\ Janson, T.\ {\L}uczak, A.\ Ruci\'nski, \emph{Random Graphs}, Wiley, New York, 2000.

\bibitem{KS15} J.\ Koml\'os and E.\ Szemer\'edi, Limit distributions for the existence of Hamilton circuits in a random graph, \emph{Discrete Mathematics} \textbf{43} (1983) 55--63.

\bibitem{Krivelevich_book} M.~Krivelevich, K.~Panagiotou, M.~Penrose and C.~McDiarmid, Random graphs, Geometry and Asymptotic Structure, Cambridge University Press 2016.

\bibitem{L15} S.\ Letzter, Path Ramsey number for random graphs, \emph{Combinatorics, Probability and Computing}, to appear.

\bibitem{KSVW16} M.\ Krivelevich, B.\ Sudakov, V.\ Vu and N.C.\ Wormald, Random regular graphs of high degree, \emph{Random Structures and Algorithms} \textbf{18} (2001) 346--363.

\bibitem{vacuum_MN} M.E.~Messinger, R.J.~Nowakowski, The Robot cleans up, \emph{Journal of Combinatorial Optimization},
\textbf{18} (2009) 350--361.

\bibitem{MPW11}
T.~M{\"u}ller, X.~P{\'e}rez-Gim{\'e}nez, and N.~Wormald.
\newblock Disjoint {H}amilton cycles in the random geometric graph.
\newblock {\em J. Graph Theory}, 68(4):299--322, 2011.

\bibitem{PS} A.~Panconesi and A.~Srinivasan, Randomized distributed edge coloring via an extension of the Chernoff-Hoeffding bounds, \emph{SIAM J.\ Comput.}\ \textbf{26} (1997) 350--368.

\bibitem{P14} A.\ Pokrovskiy, Partitioning edge-coloured complete graphs into monochromatic cycles and paths, \emph{J.\ Combin.\ Theory Ser.\ B} \textbf{106} (2014) 70--97.

\bibitem{Posa} L.\ P\'osa, Hamiltonian circuits in random graphs, \emph{Discrete Mathematics} \textbf{14} (1976) 359--364.

\bibitem{Pittel99} B.\ Pittel, J.\ Spencer, and N.\ Wormald, Sudden emergence of a giant k-core in a random graph, \emph{J.\ Combinatorial Theory, Series B} \textbf{67} (1996), 111--151.

%\bibitem{Pittel} B.\ Pittel, Note on the heights of random recursive trees and random $m$-ary search trees, \emph{Random Structures and Algorithms} \textbf{5} (1994) 337--347.

\bibitem{RW18} R.W.\ Robinson and N.C.\ Wormald, Almost all regular graphs are Hamiltonian, \emph{Random Structures and Algorithms} \textbf{5} (1994) 363--374.

\bibitem{Wormald-DE} N.C.\ Wormald, The differential equation method for random graph processes and greedy algorithms. Lectures on Approximation and Randomized Algorithms, eds.\ M. Karo\'nski and H.J. Pr\"{o}mel, PWN, Warsaw, pp. 73--155, 1999.

\bibitem{Yule} G.U.~Yule, A Mathematical Theory of Evolution, based on the Conclusions of Dr. J. C. Willis, F.R.S, \emph{Philosophical Transactions of the Royal Society B} \textbf{213(402-410)} (1925)  21--87.

\end{thebibliography}
\end{document}